\documentclass[reqno]{amsart}

\usepackage[british]{babel}
\usepackage{amssymb}
\usepackage{amsmath}
\usepackage{graphicx}
\usepackage{color}
\usepackage{tikz}
\usepackage{pgfplots, pgfplotstable}
\usepgfplotslibrary{groupplots}
\usepackage{enumitem}
\usepackage{hyperref}
\usepackage{fp-eval}
\usepackage{esint}
\usepackage{subfigure}
\usepackage{caption}
\usepackage{bm}
\usepackage{url}

\usepackage{stmaryrd}
\usepackage{booktabs}
\usepackage{commath}
\usepackage{todonotes}

\usepackage{siunitx}
\definecolor{darkblue}{rgb}{0,0,.7}
\hypersetup{colorlinks=true,linkcolor=darkblue,citecolor=darkblue}

\newlist{alphenum}{enumerate}{1}
\setlist[alphenum]{fullwidth,label={(\alph*)}}

\allowdisplaybreaks[0]
\theoremstyle{definition}
\newtheorem{theorem}{Theorem}[section]

\newtheorem{remark}[theorem]{Remark}
\newtheorem{lemma}[theorem]{Lemma}

\newcommand{\R}{\mathbb{R}}

\newcommand{\BDM}{\mathrm{BDM}}

\newcommand{\Hdiv}{H(\divergence,\Omega)}

\newcommand{\jump}[1]{[\![#1]\!]}

\newcommand{\Curl}{\curl}

\newcommand{\vecb}[1]{\bm{#1}}
\newcommand{\bvecb}[1]{\bar{\bm{#1}}}

\newcommand{\divergence}{\operatorname{div}}
\renewcommand{\div}{\operatorname{div}}
\newcommand{\curl}{\operatorname{curl}}
\newcommand{\id}{\mathop{\mathbf{id}}} 

\newcommand{\primalVh}{{\color{black}\vecb{\bar V}_h}}
\newcommand{\primalQh}{{\color{black}\bar Q_h}}
\newcommand{\primaluh}{{\color{black}\vecb{\bar u}_h}}
\newcommand{\primalvh}{{\color{black}\vecb{\bar v}_h}}
\newcommand{\primalqh}{{\color{black}\bar q_h}}
\newcommand{\primalph}{{\color{black}\bar p_h}}

\newcommand{\eqflux}{{\color{black}\sigma_h}}

\newcommand{\pseudostress}{{\color{black}\widetilde{\sigma}}}
\newcommand{\identity}{{\color{black}I_{d \times d}}}

\begin{document}

\title[Pressure-robust guaranteed error control for Stokes equations]{Guaranteed upper bounds for the velocity error of pressure-robust Stokes discretisations}

\author
{P.L. Lederer} 
\email{philip.lederer@tuwien.ac.at}
\address{ Institute for Analysis and Scientific Computing, Wiedner
 Hauptstra{\ss}e 8-10, 1040 Wien, Austria} 
 
 \author{C. Merdon}
 \email{christian.merdon@wias-berlin.de}
 \address{Weierstrass Institute for Applied Analysis and Stochastics,
Mohrenstr. 39, 10117 Berlin, Germany}

\begin{abstract}
  {This paper aims to improve guaranteed error control for the Stokes problem
  with a focus on pressure-robustness, i.e.\ for discretisations that
  compute a discrete velocity that is independent of the exact
  pressure. A Prager--Synge type result relates the velocity errors of
  divergence-free primal and perfectly equilibrated dual mixed methods
  for the velocity stress.
The first main result of the paper is a framework with
relaxed constraints on the primal and dual method.
This enables to use a recently developed mass
conserving mixed stress discretisation for the design of equilibrated fluxes
and to obtain pressure-independent guaranteed upper bounds for any
pressure-robust (not necessarily divergence-free) primal
discretisation. The second main result is a provably efficient local design of the
equilibrated fluxes with comparably low numerical costs.
Numerical examples verify the theoretical findings and show that
efficiency indices of our novel guaranteed upper bounds
are close to one.}

\end{abstract}
\maketitle

\section{Introduction}


In recent years many pressure-robust discretisations for the Stokes
equations were found and propagated that avoid a consistency error
that is connected to a relaxation of the divergence constraint
\cite{sirev2017,doi:10.1137/17M1138078,Lin14,LM2016,2016arXiv160903701L,10.1093/imanum/drz022,zhao2020pressure}.
In non-pressure-robust discretisation this consistency error can cause
severe discretisation errors in presence of large irrotational forces
in the right-hand side forcing or, more importantly, in the material
derivative of the full Navier--Stokes equations
\cite{SMAI-JCM_2019__5__89_0, ahmed2020pressurerobust}. Pressure-robustness is achieved by
using divergence-free finite element methods like
\cite{scott:vogelius:conforming,GN14b,FN13}, but also many classical
non-divergence-free methods can be turned pressure-robust by employing
a reconstruction operator on the test functions when pairing them with
irrotational forces \cite{sirev2017,LM2016,2016arXiv160903701L}. As a result, a pressure-robust method allows
pressure-independent a priori velocity error estimates.

In terms of a posteriori error control the pressure-robustness
property of a scheme also allows, in principle, separate error control
of the velocity alone and adaptive mesh refinement that is not
polluted by a concentration on the pressure error. However, this
requires that the evaluation of the a posteriori error estimator
itself is pressure-independent. There is a long history on a
posteriori error control for the Stokes problem
\cite{MR993474,carstensen2012,MR2164088,ccmerdon:nonconforming2,MR3064266,MR2995179,Repin2010}
but none of them can be considered to be fully suitable for the error
control of the velocity error of a pressure-robust discretisation of
the Stokes problem. Recent refined residual-based approaches in the
spirit of \cite{Lederer2019,kanschat2014} enable a pressure-robust
error control of the velocity error by applying the \(\curl\) operator
to the residual, following the general idea that the velocity is only
determined by the underlying vorticity equation. 

In this paper we turn
our interest to guaranteed error control for the velocity and thereby
refine existing approaches in
\cite{MR2995179,Repin2010,bringmann2016,MR3366087,merdon-thesis} that
can become inefficient when used to estimate the velocity error of
pressure-robust discretisations.

For simplicity, consider the Stokes problem in \(d\) dimensions
\begin{alignat*}{2}
  - \nu \Delta \vecb{u} + \nabla p & = \vecb{f} &\quad \textrm{on } \Omega,\\
    \divergence(\vecb{u}) & = 0 &\quad \textrm{on } \Omega.
\end{alignat*} 
with homogeneous Dirichlet boundary data and some right-hand side forcing $\vecb{f}$. Here $\vecb{u}$ and $p$
denote the velocity and pressure, respectively, and $\nu$ is the
kinematic viscosity. There are situations where $\vecb{f}$ can have a
large irrotational forcing, e.g.\ an approximation to the the material
derivative $\vecb{u}_t + (\vecb{u} \cdot \nabla) \vecb{u}$. Hence,
looking at the unrescaled Stokes model problem with $\nu \neq 1$ is
reasonable and numerical discretisations that are robust with respect
to large pressures or small viscosities $\nu$ are desirable.

Coming back to a posteriori error control, a unified approach as e.g.\ in \cite{carstensen2012,MR2995179}
rewrites many second order elliptic problems  into the form
\begin{align}\label{intro:generalform}
 - \divergence (\pseudostress) = \vecb{f} \quad \text{on } \Omega.
\end{align}
which is also possible for the Stokes problem above utilizing the
pseudo-stress \( \pseudostress := \nu \nabla \vecb{u} - p\identity\),
where $\identity$ is the $d$-dimensional identity matrix. Hence, the
application of the classical a posteriori error estimators for
(vector-valued) Poisson problems, in particular guaranteed upper
bounds like
\cite{Destuynder99expliciterror,Luce2004ALA,Braess:2008,doi:10.1137/130950100,APosterioriEstimatesforPartialDifferentialEquations,gedicke2020polynomialdegreerobust},
also work for the Stokes problem. However, this approach in general
leads to pressure-dependent velocity error estimators or estimators
that are only efficient with respect to the combined velocity and
pressure error. 

Going one step back there is also the famous Prager--Synge theorem \cite{MR0025902,bertr2019pragersynge} (originally for linear elasticity) that is nothing else than a Pythagoras theorem in $L^2$-norms, i.e.\
  \begin{align*}
    \| \nabla(\vecb{u} - \primaluh) \|_{L^2(\Omega)}^2 + \nu^{-1} \| \sigma - \eqflux \|_{L^2(\Omega)}^2
    = \| \nabla \primaluh - \nu^{-1} \eqflux \|_{L^2(\Omega)}^2,
  \end{align*}
  where \(\primaluh\) and \(\eqflux\) can be understood as some (not necessarily discrete) approximations to \(\vecb{u}\) and its stress \(\sigma := \nu \nabla \vecb{u}\).
  If both approximations are known the quantity on the right-hand side yields an a posteriori error estimator for both errors on the left-hand side.
  However, \(\primaluh\) and \(\eqflux\) have to satisfy some properties for the equality above to hold. In our Stokes setting it is required that $\vecb{u},\primaluh \in \vecb{V}_0$ and
  \begin{align}\label{intro_eqcondition}
    \int_\Omega (\divergence (\eqflux) + \vecb{f}) \cdot\vecb{v} \dif x = 0 
    \quad \text{for all } \vecb{v} \in \vecb{V}_0
  \end{align}
  where \(\vecb{V}_0\) is the subspace of divergence-free
  \(\vecb{H}^1_0(\Omega)\) functions. In comparison with
  \eqref{intro:generalform} where the divergence constraint for the
  fixed pseudo-stress has to hold pointwise for the whole space
  \(\vecb{H}^1_0(\Omega)\), equation \eqref{intro_eqcondition} is not
  only satisfied by $\eqflux$ but also by $\eqflux - q \identity$ for
  any $q \in H^1(\Omega)$ due to \(\int_\Omega \nabla q \cdot \vecb{v} \dif x
  = - \int_\Omega q \divergence (\vecb{v}) \dif x = 0 \) for all \(\vecb{v} \in
  \vecb{V}_0\). Hence, the stress approximation \(\eqflux\) can be
  gauged by any gradient force to e.g.\ mimic the pseudo-stress
  \(\pseudostress\).

  Consider now some discrete velocity \(\primaluh\) and pressure \(\primalph\) from some possibly pressure-robust discretisation of the Stokes equations.
  Many equilibration error estimators, see e.g.\ \cite{MR2995179} where unfortunately only \(\nu = 1\) is examined, employ the discrete pseudo-stress and fix the gauging freedom by the discrete pressure \(q = \primalph\), e.g.\ they compute \(\pseudostress_h\) with
  \begin{align}\label{intro:discrete_eqconstraints}
   \divergence(\pseudostress_h) + \vecb{\pi}_{k-1} \vecb{f} =
   \divergence(\eqflux + \primalph \identity) + \vecb{\pi}_{k-1} \vecb{f} = 0
  \end{align}
  where \(\vecb{\pi}_{k-1}\) is the $L^2$ best-approximation into the piecewise polynomials of order \(k-1\) and \(k\) is related to the expected rate of the primal method. Following e.g.\ \cite{MR2995179}, one can compute equilibrated
  fluxes \(\pseudostress_h\) with this constraint that are close to
  the discrete pseudo-stress \(\nu \nabla \primaluh + \primalph
  \identity\) and would get the guaranteed upper bound
  \begin{align}\label{intro:classical_estimator}
    \| &\nabla(\vecb{u} - \primaluh) \|^2_{L^2(\Omega)} \nonumber\\
    &\leq \nu^{-1}\sum_{T \in \mathcal{T}} \left( \frac{h_T}{\pi} \| (\id - \vecb{\pi}_{k-1}) \vecb{f}\|_{L^2(T)} +  \| \pseudostress_h - \primalph \identity - \nu \nabla \primaluh \|_{L^2(T)} \right)^2
  \end{align}
  which up to the oscillation term resembles the Prager--Synge calculus for \(\eqflux = \pseudostress_h - \primalph \identity\).
  However, the numerical examples below demonstrate that the dependence of this error estimator on the discrete pressure can cause
  arbitrarily large efficiency indices
  $ \eta/\| \nabla(\vecb{u} -\primaluh) \|$
  in pressure-dominant
  situations where $\vecb{f}$ has a large irrotational part, which
  can also affect the quality of the adaptive mesh refinement.
  
  
  
  In this paper we propose a novel equilibration design that avoids the pseudo-stress
  approach altogether and instead ensures the Prager--Synge equilibration condition
  \eqref{intro_eqcondition} for an $H(\divergence)$-conforming subspace of \(\vecb{V}_0\).
  This is done with the
  help of the recently developed mass conserving mixed stress
  formulation \cite{10.1093/imanum/drz022}. The resulting error
  estimator
  for a divergence-free discretisation,
  with a comparable approximation order \(k\) for the equilibrated fluxes, structurally looks
  very similar to \eqref{intro:classical_estimator}, but consists of
  the terms
  \begin{align*}
    \| &\nabla(\vecb{u} - \primaluh) \|^2_{L^2(\Omega)} \\
    &\leq \nu^{-1} \sum_{T \in \mathcal{T}} \left( c h_T^2 \| (\id - \vecb{\pi}_{k-2}) \curl (\vecb{f}) \|_{L^2(T)} +\| \mathrm{dev}(\eqflux - \nu \nabla \primaluh) \|_{L^2(T)} \right)^2,
   \end{align*}
   where \(\mathrm{dev}(A):= A -
\frac{\mathrm{tr}(A)}{d} \identity,\) is the deviatoric part of a matrix \(A \in
\mathbb{R}^{d \times d}\),
and $\mathrm{tr}(A)$ is its matrix trace. Note
that there is no dependency on the pressure or the irrotational part
of \(\vecb{f}\). The unfortunately unknown constant \(c\) stems from
approximation properties of commuting interpolators and only depends
on the shape of the cells in the triangulation. A more general result
also shows guaranteed upper bounds for non-divergence-free but
pressure-robust discretisations by accounting for the additional
divergence of the discrete velocity. The final main result of the
paper concerns a less costly localized pressure-robust design for the equilibrated
fluxes based on small problems on node patches, which is shown to be locally efficient.
  
  \medskip
  The rest of the paper is organised as follows.
Section~\ref{sec:preliminaries} introduces the Stokes model problem
and a related Prager--Synge-type theorem.
Section~\ref{sec:probust_discretisations} recalls variants of
pressure-robust discretisations of the Stokes problem in the primal
formulation. After shortly summarising classical equilibration error
estimator approaches in Section~\ref{sec:suboptimal_equilibration},
Section~\ref{sec:relaxedcontrol} proves a novel framework for
pressure-independent guaranteed upper bounds in the spirit of the
Prager--Synge theorem but with relaxed constraints. A global design of
suitable equilibrated fluxes for the novel framework based on a
recently developed mass-conserving mixed stress formulation is
discussed in Section~\ref{sec:mixedmethod}. The less costly local
design for the equilibrated fluxes is presented in
Section~\ref{sec:localdesign}. Section~\ref{sec:efficiency} is
concerned with the local efficiency of the local error estimator
estimators. Finally, Section~\ref{sec:numerics} shows in several
numerical examples that the novel upper bounds are indeed
pressure-independent and allow very sharp error control and optimal
adaptive mesh refinement for the velocity error of pressure-robust
discretisations.

\medskip

Throughout this work we use bold-face notation for vector-valued functions
and spaces, but stick to a standard notation for matrix-valued
functions and spaces to increase readability. We denote by
$L^2(\Omega)$ the space of square integrable functions and by
$H^s(\Omega)$ the standard Sobolev space with regularity~$s$. Of
special interest is the $H^1$ space with homogeneous boundary
conditions denoted by $H^1_0(\Omega)$. Restricted $L^2$-norms on
subsets $\omega \subset \Omega$ are denoted by $\| \cdot \|_\omega$,
while for $\omega = \Omega$ we simply write $\| \cdot \|$. The
$L^2$-inner product on $\omega$ and $\Omega$ is written as
$(\cdot,\cdot)_\omega$ and  $(\cdot,\cdot)$, respectively. For high order
Sobolev spaces we use the standard notation, hence $\| \cdot
\|_{H^s(\omega)}$ denotes the $H^s$-norm on $\omega$, and as before,
$\| \cdot \|_{H^s} = \| \cdot \|_{H^s(\Omega)}$.

\section{The Stokes model problem and a Prager--Synge theorem}\label{sec:preliminaries}
This section recalls the continuous Stokes model problem
and a characterisation of pressure-robustness. Then, a Prager-Synge theorem
for the Stokes problem as a point of departure for a posteriori error control is discussed.

\subsection{The Stokes model problem}
Given \(\vecb{f} \in L^2(\Omega)\) on some open, bounded domain
\(\Omega \subset \mathbb{R}^d\) ($d = 2,3$) with polygonal or
polyhedral boundary, the Stokes model problem with homogeneous Dirichlet boundary data
seeks a velocity \(\vecb{u} \in \vecb{V} := \vecb{H}^1_0(\Omega)\) and
some pressure \(p \in Q:= L^2_0(\Omega) = \lbrace q \in L^2(\Omega)
: (q,1) = 0\rbrace\) with
\begin{alignat*}{2}
  - \nu \Delta \vecb{u} + \nabla p & = \vecb{f} &\quad \textrm{on } \Omega,\\
    \divergence(\vecb{u}) & = 0 &\quad \textrm{on } \Omega,
\end{alignat*}          
where $\nu>0$ is the kinematic viscosity. The regularity assumptions of
 \(\vecb{u}\) and \(p\) above allow to expect a weak solution that
 satisfies
\begin{alignat}{2} 
  \label{eq::stokes_ex}
  \nu (\nabla \vecb{u}, \nabla \vecb{v}) - (p, \divergence(\vecb{v}))
  &= (\vecb{f},\vecb{v}) &\quad& \text{for all } \vecb{v} \in \vecb{V},\\
  (\divergence(\vecb{u}),q) &= 0 &\quad& \text{for all } q \in Q.
\end{alignat}

Note that the pressure acts as a Lagrange multiplier for the
divergence constraint. Within the subspace of divergence-free functions
\begin{align*}
  \vecb{V}_0 := \lbrace \vecb{v} \in \vecb{V} : \divergence(\vecb{v}) = 0 \rbrace
  			 = \lbrace \vecb{v} \in \vecb{V} :\forall q \in Q, \;(q,\divergence(\vecb{v})) = 0 \rbrace.
\end{align*}
the weak velocity solution \(\vecb{u}\) can also be characterised by
requiring \(\vecb{u} \in \vecb{V}_0\) and
\begin{align*}
  \nu (\nabla \vecb{u}, \nabla \vecb{v}) = (\sigma, \nabla \vecb{v}) = (\vecb{f},\vecb{v}) \quad \text{for all } \vecb{v} \in \vecb{V}_0,
\end{align*}
with its exact stress given by $\sigma := \nu \nabla \vecb{u}$.


\subsection{Characterising pressure-robustness}
According to the Helmholtz-Hodge decomposition, see e.g. \cite{GR86} for a proof,
any force \(\vecb{f} \in \vecb{L}^2(\Omega)\) can be uniquely decomposed into
\begin{align*}
 \vecb{f} = \nabla q + \mathbb{P} \vecb{f},
\end{align*}
with \(q \in {H}^1(\Omega) / \mathbb{R}\) and the (unique) divergence-free
Helmholtz--Hodge projector
\[\mathbb{P} \vecb{f} \in \lbrace \vecb{v} \in \vecb{L}^2(\Omega) : \forall  q \in
H^1(\Omega),
(\vecb{v}, \nabla q) = 0 \rbrace.\]

Due to \((\nabla q, \vecb{v}) = -(q,
\divergence \vecb{v}) = 0\) for all \(\vecb{v} \in \vecb{V}_0\), the
velocity solution \(\vecb{u}\) of the Stokes problem is not affected
by any gradient force. Indeed, from the above decomposition we see that 
\begin{align*}
  \nu (\nabla \vecb{u}, \nabla \vecb{v}) = (\mathbb{P}\vecb{f},\vecb{v}) \quad \text{for all }  \vecb{v} \in \vecb{V}_0.
\end{align*}
A discretisation that computes a velocity that preserves this property, and is independent of any gradient force $\nabla q$ that
is added to the right-hand side, is called pressure-robust (since the pressure gradient is part of the irrotational part of $\vecb{f}$), see
\cite{sirev2017,LM2016} for details.

\subsection{A Prager--Synge-type result for the Stokes system}
This section states a Pythagoras theorem for the Stokes system similar
to that of Prager and Synge for the Poisson model problem and the
linear elasticity problem \cite{MR0025902,bertr2019pragersynge}. The
Prager--Synge theorem relates the error of primal and  equilibrated
mixed approximations of the flux \(\nabla \vecb{u}\) (or
$\varepsilon(\vecb{u})$ in elasticity) and gives rise to guaranteed error
control by the design of equilibrated fluxes for these problems. The
analogue in the context of the Stokes model problem for the flux of
the velocity \(\sigma = \nu \nabla \vecb{u}\) reads as follows.

\begin{theorem}\label{thm:PragerSyngeStokes}
  Consider the solution \(\vecb{u} \in \vecb{V}_0(\Omega)\) of the
  Stokes equation \eqref{eq::stokes_ex}, any function \(\vecb{u}^\text{approx}
  \in \vecb{V}_0(\Omega)\) and any \(\sigma^\text{approx} \in H(\divergence,\Omega)\) with
  the equilibration constraint
  \begin{align}\label{eqn:EQcondition_pindependent}
    (\vecb{f} + \divergence (\sigma^\text{approx}),\vecb{v})= 0 
    \quad \text{for all } \vecb{v} \in \vecb{V}_0.
  \end{align}
Then, there holds the Pythagoras theorem
  \begin{align*}
    \| \nabla(\vecb{u} - \vecb{u}^\text{approx}) \|^2 + \nu^{-1} \| \sigma - \sigma^\text{approx} \|^2
    = \| \nabla \vecb{u}^\text{approx} - \nu^{-1} \sigma^\text{approx} \|^2.
  \end{align*}
\end{theorem}
\begin{proof}
  This follows directly from integration by parts and
  \begin{align*}
    \| \nabla(\vecb{u} - \vecb{u}^\text{approx}) \|^2 &+ \| \nabla\vecb{u} - \nu^{-1} \sigma^\text{approx} \|^2
    - \| \nabla \vecb{u}^\text{approx} - \nu^{-1} \sigma^\text{approx} \|^2 \\
    & = 2(\nabla\vecb{u} - \nu^{-1} \sigma^\text{approx}, \nabla(\vecb{u} - \vecb{u}^\text{approx}))\\
    & = 2\nu^{-1} (\vecb{f} + \divergence (\sigma^\text{approx}), \vecb{u} - \vecb{u}^\text{approx}) = 0.
  \end{align*}
 The last step follows from $\vecb{v} = \vecb{u} - \vecb{u}^\text{approx} \in
  \vecb{V}_0(\Omega)$ and \eqref{eqn:EQcondition_pindependent}. 
\end{proof}
  
Hence, the evaluation of the quantity on the right-hand side for known
approximations \(\sigma^\text{approx}\) and \(\vecb{u}^\text{approx}\)
yields guaranteed upper bounds for the two unknown errors on the
left-hand side of the identity. In this paper we consider
\(\vecb{u}^\text{approx}\) as an approximation from a pressure-robust
discretisation of the Stokes equations. Then, the computation of a
suitable \(\sigma^\text{approx}\) is the task of the a posteriori
error control.

In practise however, both constraints on the function
\(\vecb{u}^\text{approx}\) and on the flux \(\sigma^\text{approx}\) in
Theorem~\ref{thm:PragerSyngeStokes} are hard to realise. Therefore,
Section~\ref{sec:relaxedcontrol} derives guaranteed upper bounds for
\(\vecb{u}^\text{approx}\) that do not necessarily have to stem from a
divergence-free (but pressure-robust) discretisation based on
equilibrated fluxes \(\sigma^\text{approx}\) that satisfy a suitably
discretised version of the equilibration property
\eqref{eqn:EQcondition_pindependent}.

Section~\ref{sec:suboptimal_equilibration} discusses pseudo-stress
related equilibration conditions like \(\vecb{f} - \nabla q +
\divergence (\sigma^\text{approx}) = 0\) for some known (pressure
approximation) \(q\) as they are used by classical equilibrated flux
designs. However, this may lead to a dependency of \(p - q\) in the
efficiency estimates of the resulting error estimator and the whole
focus of pressure-robust discretisations is on how to avoid this.
Therefore Section~\ref{sec:novel_upperbounds} only replaces the
space \(\vecb{V}_0\) in \eqref{eqn:EQcondition_pindependent} by some
\(\Hdiv\)-conforming subspace.

Before that some suitable pressure-robust finite element methods to
compute  \(\vecb{u}^\text{approx}\) are revisited.


\section{Pressure-robust finite element methods for the Stokes problem}
\label{sec:probust_discretisations}
This section recalls pressure-robust discretisations for the primal
velocity-pressure formulation of the Stokes problem. 


\subsection{Notation}
Consider some triangulation \(\mathcal{T}\) of the domain \(\Omega\)
into regular simplices with vertices \(\mathcal{V}\) and faces
\(\mathcal{F}\). The subset of interior faces is denoted by
\(\mathcal{F}(\Omega)\). The diameter of a simplex $T \in \mathcal{T}$
is given by $h_T$. We extend this notation in a similar manner onto
faces and simply write $h_F$ for the diameter of a face $F \in
\mathcal{F}$. Further, if the triangulation is quasi uniform, we
abbreviate the notation and simply write $h$ for the maximum diameter
of all simplices. 

In the following let \(F \in \mathcal{F}\) be some arbitrary
face of an arbitrary element $T \in \mathcal{T}$, i.e. $F \subset
\partial T$.
On $F$ we then denote by $\vecb{n}$ the outward pointing unit normal
with respect to $T$, and 
\begin{align*}
  \vecb{a}_n &:= \vecb{a} \cdot \vecb{n}, \quad \textrm{and} \quad
  \vecb{a}_t := \vecb{a} - (\vecb{a} \cdot \vecb{n}) \vecb{n},
\end{align*}
 denote the scalar-valued normal and the vector-valued tangential
trace of some vector \(\vecb{a} \in \mathbb{R}^d\), respectively.
Further the brackets \(\jump{b}\) denote the jump on a common face $F$
of two adjacent elements of some (scalar or vector-valued) quantity
\(b\).

The space of element-wise (with respect to \(\mathcal{T}\))
polynomials of order \(k\) is denoted by \(P_k(\mathcal{T})\) and the
space of piece-wise vector-valued polynomials of order \(k\) by
\(\vecb{P}_k(\mathcal{T})\). The $L^2$ best-approximation into
\(P_k(\mathcal{T})\) or \(\vecb{P}_k(\mathcal{T})\) reads as $\pi_k$
or $\vecb{\pi}_k$, respectively. 

The spaces
\begin{align*}
  \mathrm{RT}_{k}(\mathcal{T})
  &:= \lbrace \vecb{v}_h \in H(\divergence,\Omega) : \forall T \in \mathcal{T} \, \exists \vecb{a}_T \in \vecb{P}_{k}(T), b_T \in P_k(T), \, \\
  & \hspace{6cm}\vecb{v}_h|_T(\vecb{x})
  = \vecb{a}_T + b_T \vecb{x}\rbrace,\\
  \mathrm{BDM}_{k}(\mathcal{T})
  &:= H(\divergence,\Omega) \cap \vecb{P}_{k}(\mathcal{T})
\end{align*}
read as the Raviart--Thomas and Brezzi--Douglas--Marini functions of
order \(k \geq 0\). Moreover,
\begin{align*}
  \mathcal{N}_{k}(\mathcal{T})
  := \{ \vecb{v}_h \in H(\curl,\Omega) : \forall T \in \mathcal{T} \, \exists \vecb{a}_T \in \vecb{P}_{k}(T), &B_T \in P_k(T)^{d\times d, \operatorname{skw}},\\
  &\vecb{v}_h|_T(\vecb{x}) = \vecb{a}_T + B_T \vecb{x} \}
\end{align*}
denote the space of N\'ed\'elec functions of order \(k \geq 0\), where
$P_k(T)^{d\times d, \operatorname{skw}}$ denotes the space of skew
symmetric matrix-valued polynomials of order $k$. Note that the space
$\mathcal{N}_{k}(\mathcal{T})$ is in the literature, see for example
\cite{MonkBook}, sometimes also called the N\'ed\'elec space of order
$k+1$ (and not $k$ as in this work). However, we used this notation so
that it matches with the definition of the Raviart--Thomas space. For
the analysis we denote by $H^s(\mathcal{T})$ the broken Sobolev space
of order $s$ with respect to the triangulation $\mathcal{T}$, i.e. 
\begin{align*}
  H^s(\mathcal{T}) := \{v \in L^2(\Omega), \forall T \in \mathcal{T}, v|_T \in H^s(T) \},
\end{align*}
with the corresponding norm $\|\cdot\|_{H^s(\mathcal{T})}$, and extend
this notation also to vector-valued versions. 

Finally we introduce the
notation $a \lesssim b$ if there exists a constant $c$ independent of
$a$ and $b$, the viscosity $\nu$ and the meshsize $h_T$ such that $a
\le c b$.

\subsection{Pressure-robust primal discretisations}
This section revisits divergence-free and pressure-robust finite element
methods in a common framework.
Consider an inf-sup stable pair of finite element spaces \(\primalVh
\subset \vecb{V}\) and \(\primalQh \subset Q\) and the associated
discrete Stokes problem: Find
$(\primaluh,\primalph) \in \primalVh \times \primalQh$ such that
\begin{subequations}  \label{primalform}
\begin{align} 
  \nu (\nabla \primaluh, \nabla \primalvh) - (\divergence(\primalvh), \primalph) &= (\vecb{f},\mathcal{R}(\primalvh)) &&  \text{for all } \primalvh \in \primalVh\\
(\divergence(\primaluh), \primalqh) &= 0 && \text{for all } \primalqh \in \primalQh.
\end{align}
\end{subequations}

\begin{table}
  \centering
  \begin{tabular}{ccccc}
    $\primalVh $& $\primalQh $& abbreviation & $r$ & $\mathcal{R}$ \\
    \toprule
    $\vecb{P}_{2}(\mathcal{T}) \cap \vecb{H}_0^1(\Omega) $ & $  P_0(\mathcal{T})$ & P20 & $1$ & $I_{\BDM_1}$ \\
    $\vecb{P}_{3}(\mathcal{T}) \cap \vecb{H}_0^1(\Omega) $ & $ P_1(\mathcal{T})$ & P31 & $2$ & $I_{\BDM_2}$ \\
    $\vecb{P}_{2+}(\mathcal{T})\cap \vecb{H}_0^1(\Omega) $ & $  P_1(\mathcal{T})$ & P2B & $2$ & $I_{\BDM_2}$ \\
    $\vecb{P}^{\text{3d}}_{2,+}(\mathcal{T}) \cap \vecb{H}_0^1(\Omega)$ & $ P_1(\mathcal{T})$ & P2B-3d & $2$ & $I_{\BDM_2}$ \\
    $\vecb{P}_{2}(\mathcal{T}) \cap \vecb{H}_0^1(\Omega)$ & $P_1(\mathcal{T})$ & SV & $2$ & $\id$ \\
    \bottomrule
  \end{tabular}
  \caption{Considered inf-sup stable Stokes pairs including the
  expected order of convergence and the used reconstruction
  operator.} \label{infsuppairs}
\end{table}

Here, \(\mathcal{R}\) denotes some reconstruction operator that
enables pressure-robustness by mapping discretely divergence-free
functions to exactly divergence-free ones. For this and
in view of the expected optimal convergence rate $r$ of the velocity ansatz space,
the reconstruction operator has to satisfy, for all
$\primalvh \in \primalVh, \primalqh \in \primalQh$,
\begin{subequations}
\begin{align} 
 \divergence(\mathcal{R}(\primalvh)) &\in \primalQh   \label{rec_two} \\
 (\divergence(\primalvh), \primalqh) &= (\divergence(\mathcal{R}(\primalvh)), \primalqh) \label{rec_two_b},\\
  (\vecb{f}, \primalvh - \mathcal{R}(\primalvh)) &=   (\vecb{f} - \vecb{\pi}_{r-2}\vecb{f}, \primalvh - \mathcal{R}(\primalvh)) \label{rec_one}.
\end{align}
\end{subequations}
Furthermore, for the local design of the equilibrated fluxes in Section~\ref{sec:localdesign}, we need that the reconstruction operator does not alter continuous linear polynomials, i.e. 
\begin{align} \label{rec_three}
   \mathcal{R}(\primalvh) = \primalvh \quad \text{for all } \primalvh \in \vecb{P}_{1}(\mathcal{T}) \cap \vecb{H}^1_0(\Omega).
\end{align}
Note that for discontinuous pressure discretisations all required
assumptions are satisfied by the standard $\text{BDM}_r$ interpolation
operator denoted by $I_{\textrm{BDM}_r}$.

\medskip
Examples for suitable finite element spaces and corresponding
reconstruction operators can be found in
\cite{sirev2017,Lin14,LMT15,LM2016}. For any divergence-free choice,
like the Scott--Vogelius (SV) finite element, no reconstruction
operator is needed and one can set \(\mathcal{R} = \id\).
Table~\ref{infsuppairs} lists suitable finite elements and their
respective reconstruction operators that are used for our numerical
experiments in Section~\ref{sec:numerics}. Here, $ \vecb{P}_{2,+}$
denotes the space of vector-valued polynomials of order $2$ including
the local cubic element bubbles, i.e. $\vecb{P}_{2,+}(\mathcal{T}) :=
\{q \in \vecb{P}_{3}(\mathcal{T}):  \forall F \in \mathcal{F}, q|_F
\in \vecb{P}_{2}(F) \}.$ In three dimensions, we similarly denote by
$\vecb{P}^{\text{3d}}_{2,+}$ the space of vector-valued polynomials of
order $2$ including the local element bubbles of order 4 and the cubic
face bubbles of order 3. A precise definition is given in example
8.7.2 in \cite{Boffi2008Mixed}. Note that we only consider a
discontinuous pressure approximation in this work since this allows an
element-wise reconstruction operator, see for example
\cite{sirev2017}. However, reconstruction
operators for continuous pressure approximations are also possible but
demand a more complicated construction and a slightly different property
\eqref{rec_one}, see \cite{2016arXiv160903701L} for details.



The following pressure-robust a priori error estimate for the velocity can be expected, see e.g.\ \cite{sirev2017} for a proof.

\begin{theorem}[Pressure-robust a priori error estimates]
  Assume that the velocity solution of the Stokes equations
  \eqref{eq::stokes_ex} fulfills the regularity \(\vecb{u}
  \in \vecb{H}^{m}(\Omega) \cap \vecb{V}_0\) with $m \geq 2$, and let
  $\primaluh$ be the discrete solution of  \eqref{primalform}. There
  holds the error estimate
  \begin{align*}
    \| \nabla(\vecb{u} - \primaluh) \| \lesssim \inf\limits_{\primalvh \in \primalVh} \| \nabla(\vecb{u} - \primalvh)\|
    + h \| (\id - \vecb{\pi}_{r-2}) \Delta \vecb{u} \| \lesssim h^s \| \vecb{u} \|_{H^{s+1}}
  \end{align*}
  where \(s := \min \lbrace m-1,r \rbrace\).
\end{theorem}

\begin{remark}
There are also some quasi-optimal a priori error estimates under weaker regularity assumptions, see \cite{0x003b8d45}.
\end{remark}


\section{Drawbacks of classical equilibrated
fluxes}\label{sec:suboptimal_equilibration} 



This section discusses classical non-pressure-robust approaches which
yield guaranteed upper bounds, but with deteriorating efficiency in pressure-dominant situations. 


\subsection{State of the art of classical non-pressure-robust flux equilibration}

In this section we shortly recall state-of-the-art equilibration error
estimators for the Stokes problem from \cite{MR2995179} in view of
Theorem~\ref{thm:PragerSyngeStokes}. For this consider the pseudo-stress reformulation
of the Stokes problem
\begin{align*}
  \pseudostress & := \nu \nabla \vecb{u} - p \identity \quad \text{and} \quad 
  \divergence \pseudostress + \vecb{f} = 0.
\end{align*}
Note that the flux \(\sigma^\text{approx} = \pseudostress + p
\identity = \sigma\) is equilibrated in the sense of
\eqref{eqn:EQcondition_pindependent} and in fact is the optimal choice
for estimating the error between \(\vecb{u}\) and any approximation
\(\vecb{u}^\text{approx} = \primaluh\) in the sense of
Theorem~\ref{thm:PragerSyngeStokes}. However, since \(\vecb{u}\) and
\(p\) are unknown one can instead approximate discrete pseudo
stress variants like
\begin{align*}
  \pseudostress_h & = \nu \nabla \primaluh - q_2 \identity \quad \text{and} \quad 
  \vecb{f} - \nabla q_1 + \divergence(\pseudostress_h + q_2 \identity) = 0
\end{align*}
with (pressure approximations) \(q_1 \in H^1(\Omega)\) or \(q_2 \in
L^2(\Omega)\). For the approximation of \(\pseudostress_h\) one can
employ standard mixed methods for Poisson problems, see e.g.\
Remark~\ref{remark::naive}. In practise
one resorts to the choices \(q_1 = 0\) and \(q_2 = \primalph\) (see
e.g.\ \cite[Theorem 4.1]{MR2995179}) or to \(q_1 = \primalph \in
H^1(\Omega)\) and \(q_2 = 0\) (see e.g.\ \cite[Corollary
5.1]{MR2995179}). The following theorem summarizes the resulting error
estimators.

\begin{theorem}\label{thm:classicalEQ} Consider the discrete Stokes solution
\((\primaluh,\primalph) \in \vecb{H}^1_0(\Omega) \times
L^2_0(\Omega)\) of an inf-sup stable discretisation on some
triangulation \(\mathcal{T}\) with inf-sup constant \(c_0 > 0\)
and its discrete stress \(\bar \sigma_h := \nu \nabla \primaluh\).
For any (pseudo-stress approximation)
\(\pseudostress_h \in H(\divergence,\Omega)\) with
\begin{align}\label{eqn:discrete_pseudostress_eqconstraint}
  \int_T \vecb{f} - \nabla q_1 + \divergence (\pseudostress_h) \dif x = 0
  \quad \text{for all } T \in \mathcal{T} 
\end{align}
and some \(q_1 \in H^1(\Omega)\) and \(q_2 \in
L^2(\Omega)\), it holds $\| \nabla \vecb{u} - \nabla  \primaluh \| \le \widetilde \eta
( \pseudostress_h)$ with the estimator 
  \begin{align*}
    \widetilde \eta
    ( \pseudostress_h)^2    
    & \leq \nu^{-2} \sum_{T \in \mathcal{T}} \Big( \frac{h_T}{\pi} \| \vecb{f} - \nabla q_1 + \divergence (\pseudostress_h )\|_{T} + \| \pseudostress_h + q_2 \identity - \bar \sigma_h \|_{T} \Big)^2 \\
    &\hspace{7cm}+  c_0^{-2}  \| \divergence (\primaluh) \|^2.
  \end{align*}
\end{theorem}
\begin{proof} 
The point of departure is the well-known error split
\cite{MR2164088,MR2995179,ccmerdon:nonconforming2}
\begin{align*}
\| \nabla \vecb{u} - \nabla \primaluh \|^2 \leq \nu^{-2} \| \vecb{r} \|^2_{\vecb{V}_0^\star} + c_0^{-2} \| \divergence (\primaluh) \|^2
\end{align*}
with the dual norm
\(
  \| \vecb{r} \|_{\vecb{V}_0^\star} := \sup_{\vecb{v} \in \vecb{V}_0 \setminus \lbrace 0 \rbrace}
  \vecb{r}(\vecb{v}) / \| \nabla \vecb{v} \|
\)
of the residual
  \begin{align*}
    \vecb{r}(\vecb{v}) & = 
    (\vecb{f} ,\vecb{v} )
    - ( \bar \sigma_h, \nabla \vecb{v})
    = 
    ( \vecb{f} + \divergence (\pseudostress_h) , \vecb{v})
    + (\pseudostress_h - \bar \sigma_h,\nabla \vecb{v} ).
  \end{align*}
  
  Due to \( (\nabla q_1, \vecb{v}) = 0\), we can add \(\nabla q_1\) 
  under the first integral
  and then \eqref{eqn:discrete_pseudostress_eqconstraint} allows to
  subtract the best-approximation $\vecb{\pi}_0
  \vecb{v}$ into \(\vecb{P}_0(\mathcal{T})\) of $\vecb{v}$ in the first term and employ piece-wise
  Poincar\'e inequalities to obtain
\begin{align*}
   (\vecb{f} - \nabla q_1 + \divergence (\pseudostress_h), \vecb{v})
  & =
   (\vecb{f} - \nabla q_1 + \divergence (\pseudostress_h),\vecb{v} - \vecb{\pi}_0 \vecb{v})\\
  & \leq \sum_{T \in \mathcal{T}} \| \vecb{f} - \nabla q_1 + \divergence (\pseudostress_h) \|_{T} \| \vecb{v} - \vecb{\pi}_0 \vecb{v} \|_{T}\\
  & \leq \sum_{T \in \mathcal{T}} \frac{h_T}{\pi} \| \vecb{f} - \nabla q_1 + \divergence (\pseudostress_h) \|_{T} \| \nabla \vecb{v} \|_{T}.
\end{align*}
Since \( (q_2 \identity , \nabla \vecb{v} ) = 0\), the second term is estimated by
\begin{align*}
  (\pseudostress_h - \bar \sigma_h, \nabla \vecb{v})
  & =  (\pseudostress_h + q_2 \identity - \bar \sigma_h, \nabla \vecb{v})
    \leq \sum_{T \in \mathcal{T}} \| \pseudostress_h + q_2 \identity - \bar \sigma_h \|_{T} \| \nabla \vecb{v} \|_{T}.
\end{align*}
A Cauchy inequality concludes the proof.
\end{proof}

\begin{remark}[Realisations] \label{remark::naive} A possible design
of \(\pseudostress_h\) involves the Raviart-Thomas or
Brezzi-Douglas-Marini finite element spaces of order \(k\) which is
denoted here by \(\vecb{\widetilde V}_h\) and its divergence space denoted
by \(\widetilde Q_h\). Then one computes \(\pseudostress_h = \pseudostress_h^\text{CEQ} \in
(\vecb{\widetilde V}_h)^d\) and \(\vecb{\widetilde u}_h \in (\widetilde Q_h)^d\) such that
\begin{align*}
  (\pseudostress_h^\text{CEQ},\tau_h) + (\vecb{u}_h, \divergence(\tau_h)) & = (\bar \sigma_h - q_2 \identity,\tau_h)
  && \quad \text{for all } \tau_h \in (\vecb{\widetilde V}_h)^d\\
  (\vecb{v}_h, \nu \divergence(\pseudostress_h^\text{CEQ})) & = - (\vecb{f} - \nabla q_1,\vecb{v}_h)
  && \quad \text{for all } \vecb{v}_h \in (\widetilde Q_h)^d
\end{align*}
As stated above, one usually takes the discrete pressure
\((q_1,q_2) = (\primalph,0)\) or \((q_1,q_2) = (0,\primalph)\)
depending on its regularity, which also enables local designs of
equilibrated fluxes as detailed in e.g.\ \cite{MR2995179} or using
component-wise designs known for elliptic problems, see e.g.\
\cite{Destuynder99expliciterror,Luce2004ALA,Braess:2008,doi:10.1137/130950100,APosterioriEstimatesforPartialDifferentialEquations,merdon-thesis,gedicke2020polynomialdegreerobust}.

\end{remark}

\begin{remark}[Efficiency]\label{rem:classical_efficiency}
 In the numerical
 example below, we show that even the best-approximation strategy from the previous Remark (which gives  a lower bound for any local equilibration in the same space)
 with \(q_1 = 0\) and \(q_2 = \primalph\) is not efficient for the
 velocity error alone in a pressure-dominant situation. 
 
 Note that this is not in contradiction with efficiency proofs for
these classical equilibration designs which usually focus on the error
of a combined velocity-pressure norm like \(\| \nabla(\vecb{u} -
\vecb{u}_h)\| + \nu^{-1} \| p - p_h \|\) (see e.g.\ \cite[Theorem
6.1]{MR2995179}). Recall that pressure-robust methods allow for a
velocity error that is independent of the pressure, while in classical
non-pressure-robust methods the velocity error can scale with the
pressure best-approximation error in pressure-dominant situations. In
such a situation, the classical error estimator, when applied to a
pressure-robust method, will still be efficient for measuring the
dominating pressure error, but not for the (much smaller) velocity
error. As the efficiency of equilibration error estimators is usually
traced back to the efficiency of the explicit estimator, the
interested reader can find a deeper discussion in \cite{Lederer2019}
for classical explicit standard-residual estimators.
 
 One way to possibly improve efficiency in these pressure-dependent designs is the
 pre-computation of a better pressure approximation \(q_2\) as it has
 been suggested e.g.\ in \cite{MR3366087}. However, in situations were
 the pressure is complicated or non-smooth this comes at highly
 increased numerical costs. 
\end{remark}

\begin{figure}
  \begin{center}
  \includegraphics[]{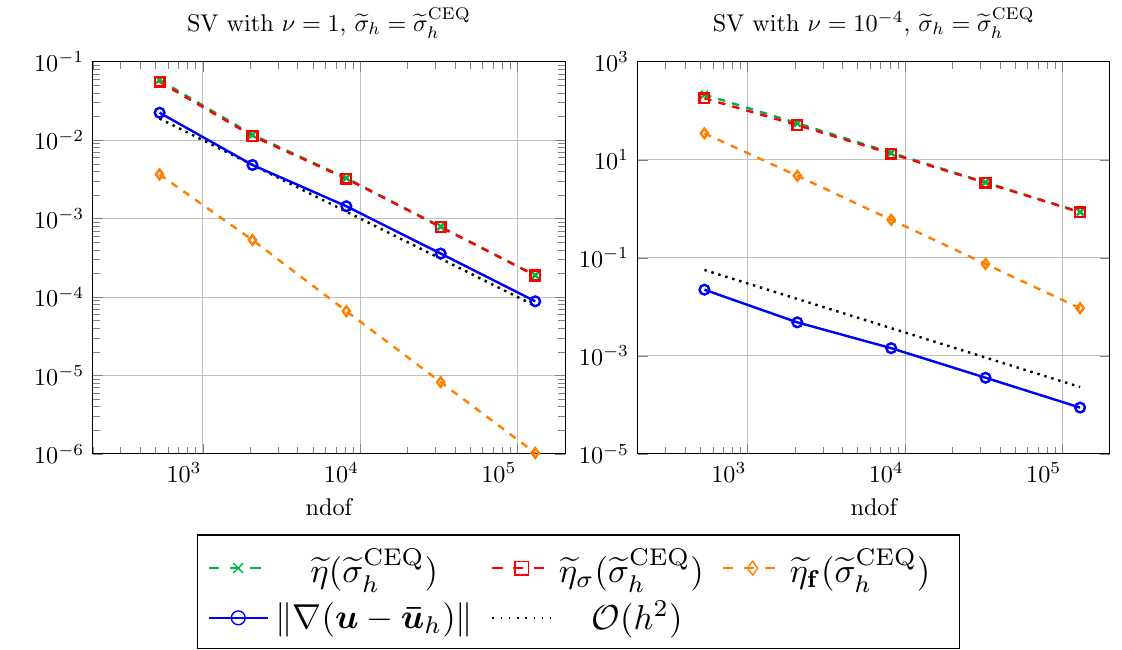}
  \vspace{5pt}
  \caption{Convergence history of the exact error and the
  classical error estimator quantities on uniformly refined meshes for
  SV with \(\nu=1\) (left) and \(10^{-4}\) (right).}
  \label{fig:exzero}
  \end{center}
\end{figure}

\subsection{Numerical example} \label{sec::numericsclassical}

In the following we demonstrate a possible deterioration of the
efficiency of classical equilibration estimators even in the case of
very smooth but pressure-dominant solutions. To this end we consider
the Stokes problem on a unit square domain \(\Omega = (0,1)^2\) with
the smooth prescribed solution
\begin{align*}
  \vecb{u}(x,y) := \curl \left(x^2(1-x)^2y^2(1-y)^2\right) \quad \text{and} \quad  p(x,y) := x^5 + y^5 - 1/3
\end{align*}
with matching right-hand side \(\vecb{f} := - \nu \Delta \vecb{u} + \nabla p\) for variable viscosity \(\nu\).

We denote by $\eta^\text{CEQ}$ the estimator of Theorem
\ref{thm:classicalEQ} where
$\pseudostress_h=\pseudostress_h^\text{CEQ}$ is the solution of the
mixed system given in Remark \ref{remark::naive} with $q_1 = 0$ and
$q_2 = \primalph$. Further we introduce the following quantities
\begin{align*}
  \widetilde \eta_{\bold f}(\pseudostress_h^\text{CEQ}) &:=(\nu\pi)^{-1} \|h_T (\bold f + \operatorname{div}(\pseudostress_h^\text{CEQ}))\|,\\
  \widetilde \eta_{\sigma}(\pseudostress_h^\text{CEQ}) &:= \nu^{-1}\|\pseudostress_h^\text{CEQ} + \primalph \identity - \bar \sigma_h\|.
\end{align*}

Figure \ref{fig:exzero} shows the error convergence history for
uniform refinement and a corresponding barycentric refinement for the
Scott--Vogelius (SV) finite element method for two different choices of \(\nu\).
Since this
method shows the expected convergence order $r = 2$, we used the
spaces $\vecb{\widetilde V}_h = (\text{BDM}_{2})^2$ and $\widetilde{Q}_h =
\vecb{P}_1$ in the  mixed system given in Remark
\ref{remark::naive}. As expected, the error estimator contribution
\(\eta^\text{CEQ}_{\bold f}\) is of higher order, but in
a pressure-dominant scenario with \(\nu = 10^{-4}\) it is, even on the
finest mesh, much larger than the exact error. Also the other
contribution \(\eta^\text{CEQ}_{\sigma}\) is much larger than the
exact error and not of higher order. In fact, the efficiency index scales approximately with
\(\nu^{-1}\). Note that also any local equilibrated fluxes from the same
space $\vecb{\widetilde V}_h$ will only lead to larger bounds than the
best-approximation in this space that we computed here. 

In summary, there are
situations (when \(\nu^{-1} (p - \primalph)\) is large compared to
\(\vecb{u}\)) were the classical equilibration designs cannot be
considered efficient for the velocity error of a pressure-robust
discretisation.
Note that since the SV finite element method provides
exactly divergence-free velocity solutions, it holds
$c_0^{-1}\|\divergence( \primaluh)\| = 0$ in
Theorem~\ref{thm:classicalEQ}.

\section{Novel pressure-robust guaranteed upper bounds}
\label{sec:novel_upperbounds} \label{sec:relaxedcontrol}

\subsection{Motivation}

The main idea of our novel relaxed equilibration is motivated by the
findings of the pressure-robust mass conserving mixed stress (MCS)
method from
\cite{10.1093/imanum/drz022,Lederer:2019c,lederer2019mass}. 
In order to apply this method for an equilibration in the sense of the
Prager-Synge Theorem~\ref{thm:PragerSyngeStokes}, we need to slightly
weaken the regularity assumptions of $\sigma^\text{approx}$ and the
equilibration constraint \eqref{eqn:EQcondition_pindependent}. To this
end we define the novel function space 
\begin{align*}
  H(\curl \divergence, \Omega) := \{ \tau \in L^2(\Omega)^{d \times d}: \divergence(\tau) \in (H_0(\divergence,\Omega))^\star, \mathrm{tr}(\tau) = 0  \},
\end{align*}
where $(H_0(\divergence,\Omega))^\star$ is the dual space of
$H_0(\divergence, \Omega) :=
\{\vecb{v} \in H(\div, \Omega): \vecb{v}_n=0 \textrm{ on } \partial
\Omega \}$. This then allows to reformulate
\eqref{eqn:EQcondition_pindependent} as
\begin{align} \label{eq::relaxed_equi_cont}
  \langle \div (\sigma^\text{approx}), \vecb{v} \rangle_{H_0(\divergence,\Omega)} + (\vecb{f},\vecb{v}) = 0 \quad \text{for all } \vecb{v} \in H_0(\divergence, \Omega) \textrm{ with } \div(\vecb{v}) = 0,
\end{align}
where $\langle \cdot, \cdot \rangle_{H_0(\divergence,\Omega)}$ denotes
the duality pair on $H_0(\divergence,\Omega)$. Following the ideas of
the MCS method we continue to derive a discrete version of
\eqref{eq::relaxed_equi_cont}. For this we define for some given \(k \geq 0\),
the discrete stress and velocity space by
\begin{align*}
 \Sigma_h
 &:= \bigl\lbrace \tau_h \in P_k(\mathcal{T})^{d \times d} : \mathrm{tr}(\tau_h) = 0; \forall F \in \mathcal{F}(\Omega),
 \jump{(\tau_h)_{nt}} = 0  \bigr\rbrace,\\
 \vecb{V}_h &:= \mathrm{RT}_{k}(\mathcal{T}) \cap H_0(\div, \Omega).
\end{align*}
Although $\vecb{V}_h \subset H_0(\div, \Omega)$, the space
$\Sigma_h$ is slightly non-conforming with respect to the
space $H(\curl \div, \Omega)$, see \cite{lederer2019mass} for details. 
To mimic the continuous duality pair $\langle \cdot, \cdot
\rangle_{H_0(\divergence,\Omega)}$ we define for all functions
$\sigma_h \in \Sigma_h$ and $\vecb{v}_h \in \vecb{V}_h$
the bilinear form  
\begin{align*}
  \langle \divergence(\sigma_h), \vecb{v}_h \rangle_{\vecb{V}_h} & := \sum_{T \in \mathcal{T}} ( \divergence(\sigma_h), \vecb{v}_h)_T
  - \sum_{F \in \mathcal{F}} ( \jump{(\sigma_h)_{nn}}, (\vecb{v}_h)_n)_F\\
  & = - \sum_{T \in \mathcal{T}} ( \sigma_h, \nabla \vecb{v}_h)_T
  + \sum_{F \in \mathcal{F}} ( (\sigma_h)_{nt} , \jump{(\vecb{v}_h)_t})_F,
\end{align*}
which can be interpreted as a distributional divergence. 

This then leads to a discrete version
of the relaxed equilibration \eqref{eq::relaxed_equi_cont} given by
\begin{align*}
  (\vecb{f},\vecb{v}_h) + \langle \divergence(\sigma_h), \vecb{v}_h \rangle_{\vecb{V}_h} = 0 \quad \text{for all } \vecb{v}_h \in \vecb{V}_h \textrm{ with } \div(\vecb{v}_h) = 0.
\end{align*}
Before we derive guaranteed upper bounds for \(\sigma_h \in \Sigma_h\)
with this constraint in Theorem~\ref{thm:EQestimator} below, some
additional tools are introduced.

\subsection{Commuting interpolation operators}

In the following we employ several commuting interpolators whose
properties are collected here. For this note that the operator
\(\Curl\) has a different definition in two and three dimensions and
further depends on the dimension of the quantity it is applied to. If
applied to some scalar-valued quantity \(\psi \in H^1(\Omega)\) it is
defined by \(\curl (\psi) := (\partial_{y} \psi, - \partial_{x}
\psi)^T\). If applied to some vector-valued quantity \(\vecb{\psi} =
(\psi_1,\psi_2) \in \vecb{H}^1(\Omega)\) for \(d=2\) it reads as
\(\curl (\vecb{\psi}) := \partial_{x} \psi_2 - \partial_{y} \psi_1\),
and if applied to some vector-valued quantity \(\vecb{\psi} \in
H(\curl,\Omega)\) for \(d=3\) it reads as \(\curl (\vecb{\psi}) :=
\nabla \times \vecb{\psi}\). Now let \(I_{\vecb{V}_h}\) be the
standard Raviart-Thomas interpolation operator. Since the de-Rham
complex (i.e. the commuting properties we aim to discuss) involves
different spaces depending on the spatial dimension, we define 
\begin{align*}
  {W}_h := \begin{cases}
    P_{k+1} \cap H^1(\Omega) & \textrm{for } d = 2 \quad (\textrm{Lagrange space of order } k+1),\\
    \mathcal{N}_{k} & \textrm{for } d = 3 \quad (\textrm{N\'ed\'elec space of order } k).
    \end{cases}
\end{align*}
Then, for $d =3$, the interpolation $ I_{W_h}$ is the standard
N\'ed\'elec interpolation operator as in \cite{Boffi2008Mixed}, and
for $d =2$ we use the (corresponding commuting) $H^1$-interpolation
operator as given in \cite{MonkBook}. 

\begin{theorem}[Commuting interpolations]\label{thm:commuting_interpolation}   
  Let \(T \in \mathcal{T}\) be an arbitrary simplex and let $F \in
\mathcal{F}$ be an arbitrary face. The operators \( I_{W_h}\) and
\(I_{\vecb{V}_h}\) enjoy the properties:
  \begin{itemize}
  \item For $d = 2$ we have the commuting property 
     \begin{align}\label{eqn:commuting_property_two}
    I_{\vecb{V}_h} \curl (\psi) = \curl ( I_{W_h} \psi)
    \quad \text{for all } \psi \in H^2(\Omega),
  \end{align}
  and the approximation properties
      \begin{align}\label{eqn:commuting_orthogonality_faces_two}    
    ( (\id -  I_{W_h}) \psi , q_h )_F &= 0
    \quad \text{for all } q_h \in P_{k-1}(F),\\
    \label{eqn:commuting_orthogonality_cells_two}    
    ( (\id -  I_{W_h}) \psi , q_h)_T &= 0
    \quad \text{for all } q_h \in P_{k-2}(T), \\
        \| \psi -  I_{W_h} \psi \|_{T} & \leq c_2 h_T \| \nabla \psi \|_{T} \quad \textrm{for all } \psi \in H^2(T).
      \end{align}
    \item For $d = 3$ we have the commuting property 
      \begin{align}\label{eqn:commuting_property_three}    
    I_{\vecb{V}_h} \curl \vecb{\psi} = \curl ( I_{W_h}\vecb{\psi})
    \quad \text{for all } \vecb{\psi} \in \vecb{H}^1(\curl,\Omega),
  \end{align}
  where $\vecb{H}^1(\curl,\Omega) = \{ \vecb{\psi} \in
  \vecb{H}^1(\Omega): \curl(\vecb{\psi}) \in \vecb{H}^1(\Omega)\}$,
  and the approximation properties
      \begin{align}\label{eqn:commuting_orthogonality_faces_three}    
    ( (\id -  I_{W_h}) \vecb{\psi}, \vecb{q}_h \times \vecb{n})_F &= 0
    \quad \text{for all } \vecb{q}_h \in \vecb{P}_{k-1}(F),\\
    \label{eqn:commuting_orthogonality_cells_three}    
    ( (\id -  I_{W_h}) \vecb{\psi},  \vecb{q}_h )_T &= 0
    \quad \text{for all } \vecb{q}_h \in \vecb{P}_{k-2}(T),\\
         \|\vecb{\psi} -  I_{W_h} \vecb{\psi} \|_{T} & \leq c_2 h_T \| \nabla\vecb{\psi} \|_{T} \quad \textrm{for all }\vecb{\psi} \in \vecb{H}^1(\curl,T).
      \end{align}
    \item For $d = 2$ and $d = 3$ we have
      \begin{align}
    (((\id -  I_{\vecb{V}_h}) \vecb{v})_n,  \; {q}_h )_F &= 0
    \quad \text{for all } {q}_h \in {P}_{k}(F), \label{eqn:commuting_orthogonality_normalortho}\\
  ( (\id-I_{\vecb{V}_h})\vecb{v}, \vecb{q}_h )_T & = 0
  \quad \text{for all } \vecb{q}_h \in \mathcal{N}_{k-2}(T), \label{eqn:commuting_orthogonality_nedelec}\\
        \| \vecb{v} - I_{\vecb{V}_h} \vecb{v} \|_{T} & \leq c_1 h_T \| \nabla \vecb{v} \|_{T} \quad \textrm{for all } \vecb{v} \in \vecb{H}^1(T),
      \end{align}
      \end{itemize}
  with constants \(c_1,c_2\) independent of \(h_T\). 
\end{theorem}
\begin{proof}
  The properties of $I_{W_h}$ in two and three dimensions follows with
  the results in \cite{MonkBook} and standard Bramble-Hilbert type
  arguments. Note that the results in \cite{MonkBook} (in two
  dimensions) are only given for the rotated commuting diagram, i.e.
  $\nabla I_{W_h} \psi = I_{\mathcal{N}_k}(\nabla \psi)$, where $
  I_{\mathcal{N}_k}$ is the standard N\'ed\'elec interpolator.
  The claimed results in this work follow immediately as the
  Raviart--Thomas space is simply a rotated
  N\'ed\'elec space and the curl is the rotated gradient, thus we have
  $(I_{\mathcal{N}_k}(\nabla \psi))^\perp = I_{\vecb{V}_h}(\Curl
  \psi)$.
Similar results can be found in \cite{DEMKOWICZ200029,Melenk2019,Boffi2008Mixed}.  
  
  Equation \eqref{eqn:commuting_orthogonality_normalortho} is a
  standard property of the Raviart-Thomas interpolation operator, see
  e.g. \cite{Boffi2008Mixed}. We continue with the proof of
  \eqref{eqn:commuting_orthogonality_nedelec} but only present the
  case $d=3$ since the two dimensional results follows with similar
  arguments. First observe that any divergence-free function
  \(\vecb{v} \in \vecb{V}_0\) has a (local) potential \(\vecb{v} =
  \curl (\vecb{\psi})\) for some \(\vecb{\psi} \in H^1(\curl,T)\).
  Then, for any \(\vecb{q}_h \in \mathcal{N}_{k-2}(T)\),
  \eqref{eqn:commuting_property_three} and an integration by parts
  show
  \begin{align*}
  ( (\id-I_{\vecb{V}_h})\vecb{v} ,\vecb{q}_h )_T
  & = ( (\id-I_{\vecb{V}_h})\curl (\vecb{\psi}) , \vecb{q}_h )_T
   = ( \curl ((\id-I_{W_h} ) \vecb{\psi}) , \vecb{q}_h )_T\\
  & = ( (\id-I_{W_h}) \vecb{\psi} ,\curl (\vecb{q}_h ))_T
    - ((\id-I_{W_h}) \vecb{\psi} ,\vecb{q}_h \times \vecb{n} )_F.
  \end{align*}
  Since \(\vecb{q}_h \in \mathcal{N}_{k-2}(T) \subset
  \mathbb{\mathcal{P}}_{k-1}(T)\) and hence \(\curl(\vecb{q}_h) \in
  P_{k-2}(T)\) and \(\vecb{q}_h \times \vecb{n}|_F \in P_{k-1}(F)\),
  the right-hand side vanishes due to
  \eqref{eqn:commuting_orthogonality_cells_three} and
  \eqref{eqn:commuting_orthogonality_faces_three}. This concludes the
  proof. 
\end{proof}

\subsection{Novel pressure-robust guaranteed upper bounds}

We are now in the position to derive pressure-robust guaranteed upper
bounds via equilibrated fluxes with a proper discrete analogue of the
equilibration constraint \eqref{eqn:EQcondition_pindependent}. 

\begin{theorem}\label{thm:EQestimator} Assume the regularity
\(\vecb{f} \in H(\curl,\Omega)\). Let $\primaluh,\primalph$ be the
solution of \eqref{primalform} and let \(\bar \sigma_h := \nu \nabla
\primaluh\). For any discrete stress \(\sigma_h \in \Sigma_h\)
that is equilibrated in the sense
\begin{align} \label{eq::relaxed_equi_disc} 
  (\vecb{f},\vecb{v}_h) + \langle \divergence(\sigma_h), \vecb{v}_h \rangle_{\vecb{V}_h} = 0 \quad \text{for all } \vecb{v}_h \in \vecb{V}_h \textrm{ with } \div(\vecb{v}_h) = 0,
\end{align}
it holds $\| \nabla(\vecb{u} - \bvecb{u} _h) \| \leq \eta(\sigma_h)$
with the error estimator 
\begin{align*}
  \eta(\sigma_h)^2 := \nu^{-2} \sum_{T \in \mathcal{T}} \Big(c_1 c_2 h_T^2  \| (\id - \vecb{\pi}_{k-2}) \curl (\vecb{f}) \|_T
    + &\| \mathrm{dev}(\sigma_h - \bar \sigma_h) \|_{T} \Big)^2 \\
    &+  c_0^{-2}  \| \divergence (\primaluh) \|^2.
\end{align*}
\end{theorem}
\begin{proof}

As in Theorem~\ref{thm:classicalEQ} the point of departure is the error split
\begin{align*}
\| \nabla(\vecb{u} - \bvecb{u} _h) \|^2 \leq \nu^{-2} \| \vecb{r} \|_{\vecb{V}_0^\star}^2 +  c_0^{-2}  \| \divergence (\primaluh) \|^2
\end{align*}
where it remains to bound the residual functional
\begin{align*}
  \vecb{r}(\vecb{v}) = ( \vecb{f},\vecb{v}) - ( \bar \sigma_h, \nabla \vecb{v}) \quad \text{for all } \vecb{v} \in \vecb{V}_0,
\end{align*}
in its dual norm
\begin{align*}
  \| \vecb{r} \|_{\vecb{V}_0^\star} := \sup_{\vecb{v} \in \vecb{V}_0 \setminus \lbrace \vecb{0} \rbrace}
  \frac{\vecb{r}(\vecb{v})}{\| \nabla \vecb{v} \|}.
\end{align*}

  Consider an arbitrary test function \(\vecb{v} \in \vecb{V}_0\) and
  some equilibrated flux \(\sigma_h\) with the properties stated
  above. The equilibration constraint \eqref{eq::relaxed_equi_disc} for $\vecb{v}_h = I_{\vecb{V}_h} \vecb{v}$
  and an integration by parts yields
  \begin{align*}
    \vecb{r}(\vecb{v}) =& \langle \vecb{f} + \divergence(\sigma_h), \vecb{v} - I_{\vecb{V}_h} \vecb{v} \rangle_{\vecb{V}_h} + (\sigma_h - \bar \sigma_h, \nabla \vecb{v} )\\
     =& 
   \sum_{T \in \mathcal{T}} (\vecb{f} + \divergence(\sigma_h), \vecb{v} - I_{\vecb{V}_h} \vecb{v})_T\\
    &+ \sum_{F \in \mathcal{F}(\Omega)} ( \jump{(\sigma_h)_{nn}}, (\vecb{v} - I_{\vecb{V}_h} \vecb{v})_n)_F
    + (\sigma_h - \bar \sigma_h,\nabla \vecb{v}).
  \end{align*}
  Since \(\jump{(\sigma_h)_{nn}} \in P_{k}(F)\), the second integral
 vanishes using to orthogonality properties of the normal flux of
 \((\vecb{v} - I_{\vecb{V}_h} \vecb{v})\), see
 \eqref{eqn:commuting_orthogonality_normalortho} of
 Theorem~\ref{thm:commuting_interpolation}.  
  The last integral on the right-hand side can be estimated by
\begin{align}\label{eqn:main_proof_estimate1}
  (\sigma_h - \bar \sigma_h, \nabla \vecb{v})
  = (\mathrm{dev}(\sigma_h - \bar \sigma_h), \nabla \vecb{v} )
   \leq \sum_{T \in \mathcal{T}} \| \mathrm{dev}(\sigma_h - \bar \sigma_h) \|_{T} \| \nabla \vecb{v} \|_{T}.
\end{align}
Here, \(\mathrm{dev}(A)\) denotes the deviatoric part of a \(A\) and it was used that \(A - \mathrm{dev}(A) = \mathrm{tr}(A) \identity/d\) is orthogonal on gradients of divergence-free functions.

The estimate of the first integral will be presented only for the case
$d=3$, as for $d=2$ the arguments are very similar. Since \(\vecb{v}
- I_{\vecb{V}_h} \vecb{v}\) is divergence-free, there exists some
\(\vecb{\psi} \in \vecb{H}^1(\Omega)\) with $\| \nabla \vecb{\psi}
\|_T \le \|\vecb{v} - I_{\vecb{V}_h} \vecb{v}\|_T$, see for example in
\cite{CostaMcInt10}, such that \(\vecb{v} - I_{\vecb{V}_h} \vecb{v} =
\curl \vecb{\psi}\) and by the interpolation properties we have
\begin{align}\label{eqn:curlpotential_norm_estimate}
	\| \nabla \vecb{\psi} \|_T \le \| \curl \vecb{\psi} \|_T = \|\vecb{v} - I_{\vecb{V}_h} \vecb{v}\|_T \leq c_1 h_T \| \nabla \vecb{v} \|_T \quad \text{on every } T \in \mathcal{T}. 
\end{align}
  
  
Moreover, it holds
$I_{\vecb{V}_h} \curl \vecb{\psi} = 0$ and hence, the commuting property
\eqref{eqn:commuting_property_three} in
Theorem~\ref{thm:commuting_interpolation} yields \(\curl
I_{W_h} \vecb{\psi}= 0\) where \(I_{W_h}\) is the matching commuting
interpolation operator. Note that the application of the operator
$I_{W_h}$ to $\vecb{\psi}$ is well defined, since locally on each
element $T \in \mathcal{T}$ we have that $\vecb{v} - I_{\vecb{V}_h}
\vecb{v} \in \vecb{H}^1(T)$ and thus we can bound $\| \nabla \curl
\vecb{\psi}\|_T \le \|\nabla(\vecb{v} - I_{\vecb{V}_h} \vecb{v}) \|_T$
which gives $\vecb{\psi} \in \vecb{H}^1(\curl, T)$.

Next, if \(k \geq 2\), consider some N\'ed\'elec function
\(\vecb{\theta}_h \in \mathcal{N}_{k-2}(\mathcal{T})\) chosen such
that \(\curl \vecb{\theta}_h = \vecb{\pi}_{k-2} \curl (\vecb{f} +
\divergence(\sigma_h))\). Then,
\eqref{eqn:commuting_orthogonality_nedelec} and the properties of
\(\sigma_h\) yield
\begin{align*}
    \sum_{T \in \mathcal{T}} (\vecb{f} + \divergence(\sigma_h), \vecb{v} - I_{\vecb{V}_h} \vecb{v})_T
   =& \sum_{T \in \mathcal{T}} (\vecb{f} + \divergence(\sigma_h) - \vecb{\theta}_h, \curl (\vecb{\psi} - I_{W_h} \vecb{\psi}) )_T\\
   = &\sum_{T \in \mathcal{T}}  (\id -  \vecb{\pi}_{k-2}) \curl (\vecb{f} + \divergence(\sigma_h)),  \vecb{\psi} - I_{W_h} \vecb{\psi})_T\\
   &+ \sum_{F \in \mathcal{F}(\Omega)} (\jump{\vecb{f} + \divergence(\sigma_h)} \times \vecb{n}, \vecb{\psi} - I_{W_h} \vecb{\psi})_F.
\end{align*}
Since \(\vecb{f} \in H(\curl,\Omega)\) and \(\divergence(\sigma_h) \in
\vecb{P}_{k-1}(\mathcal{T})\), the second integral vanishes due to
property \eqref{eqn:commuting_orthogonality_faces_three} of
\(I_{W_h}\) from Theorem~\ref{thm:commuting_interpolation}. For the
remaining term the equation $(\id - \vecb{\pi}_{k-2})\curl (
\divergence(\sigma_h)) = 0$, the interpolation properties of
\(I_{W_h}\) (see again Theorem~\ref{thm:commuting_interpolation}) and
\eqref{eqn:curlpotential_norm_estimate} yield
\begin{align*}
\sum_{T \in \mathcal{T}} ( (\id - \vecb{\pi}_{k-2}) &\curl (\vecb{f} + \divergence(\sigma_h)) , \vecb{\psi}- I_{W_h} \vecb{\psi})_T \\
& \leq \sum_{T \in \mathcal{T}} \| (\id -  \vecb{\pi}_{k-2}) \curl (\vecb{f}) \|_{T} \| \vecb{\psi} - I_{W_h} \vecb{\psi}\|_{T}\\
& \leq \sum_{T \in \mathcal{T}} c_2 h_T \| (\id -  \vecb{\pi}_{k-2}) \curl (\vecb{f}) \|_{T} \| \nabla \vecb{\psi} \|_{T}\\
& \leq \sum_{T \in \mathcal{T}} c_1 c_2 h_T^2 \| (\id - \vecb{\pi}_{k-2}) \curl (\vecb{f}) \|_{T} \| \nabla \vecb{v} \|_{T}.
\end{align*}


The last estimate, \eqref{eqn:main_proof_estimate1} and a Cauchy inequality show
\begin{align*}
  \vecb{r}(\vecb{v}) \leq \sum_{T \in \mathcal{T}} \left(c_1 c_2 h_T^2 \| (\id -  \vecb{\pi}_{k-2}) \curl (\vecb{f}) \|_{T}
  + \| \mathrm{dev}(\sigma_h - \bar \sigma_h) \|_{T} \right) \| \nabla \vecb{v} \|_{T}
\end{align*}
and hence
\begin{align*}
  \| \vecb{r} \|^2_{\vecb{V}_0^\star} \leq\sum_{T \in \mathcal{T}} \left(c_1 c_2 h_T^2 \| (\id -  \vecb{\pi}_{k-2}) \curl (\vecb{f}) \|_{T} + \| \mathrm{dev}(\sigma_h - \bar \sigma_h) \|_{T} \right)^2.
\end{align*}
This concludes the proof.
  \end{proof}

  \begin{remark}
    Theorem \ref{thm:EQestimator} also holds true in the case when we only have the local regularity assumption $\vecb{f} \in H(\curl,T)$ for all $T \in \mathcal{T}$. Note however, that this introduces another term on the boundary of the elements given by
    \begin{align*}
      c_3\sum\limits_{F \in \mathcal{F}(\Omega)} h_F^{3} \| (\id - \vecb{\pi}_{k-1})\jump{\vecb{f} \times \vecb{n}} \|_{F}^2
    \end{align*}
added to the estimator $\eta(\sigma_h)^2$ given in Theorem \ref{thm:EQestimator}. Here, $c_3$ is an additional constant that only depends on the shape of the simplices $T \in \mathcal{T}$.
\end{remark}
  

\begin{remark}[Divergence error]
  In the numerical examples of Section~\ref{sec:numerics}
  it becomes apparent that the efficiency of the error estimator is
  mostly limited by the divergence-term \(c_0^{-1} \| \divergence
  (\vecb{u}_h) \|\) for non-divergence-free discretisations. To avoid
  this term and possibly further increase the efficiency, one may
  consider a divergence-free post-processing \(\vecb{s}_h \in
  \vecb{H}^1(\Omega)\) of \(\vecb{u}_h\) and perform the error
  estimation for \(\vecb{s}_h\) or \(\bar \sigma_h := \nabla
  \vecb{s}_h\). Effectively this would replace the term \(c_0^{-1}  \|
  \divergence (\vecb{u}_h) \|\) by \(\|\nabla(\vecb{s}_h -
  \primaluh)\|\) without the possibly small constant \(c_0\).
  A candidate for such a post-processing may be a locally computed
  approximation into a divergence-free Scott--Vogelius finite element
  space (on a barycentrically refined subgrid) similar to
  \cite{kreuzer2020quasioptimal}.
\end{remark}

\section{Global Equilibration}
\label{sec:mixedmethod}
This section derives one possible design of an equilibrated flux
$\sigma_h = \sigma^\text{GEQ}_h$ for Theorem~\ref{thm:EQestimator}. The idea is to solve a
global problem with a mixed MCS method: Find \((\sigma^\text{GEQ}_h,\vecb{u}_h,p_h) \in \Sigma_h \times
\vecb{V}_h \times Q_h\) such that
\begin{subequations}\label{global:MCS}
\begin{align}
  \frac{1}{\nu}(\sigma^\text{GEQ}_h,\tau_h) + \langle \divergence(\tau_h), \vecb{u}_h \rangle_{\vecb{V}_h} & =  ( \nabla \bar{\vecb{u}}_h, \tau_h)
  && \text{for all } \tau_h \in \Sigma_h,\\
  \langle \divergence(\sigma^\text{GEQ}_h), \vecb{v}_h \rangle_{\vecb{V}_h} + (\div(\vecb{v}_h),p_h) &= (-\vecb{f},\vecb{v}_h)
  && \text{for all } \vecb{v}_h \in \vecb{V}_h,\\
  (\div(\vecb{u}_h),q_h) &=  - (\divergence( \bar{\vecb{u}}_h),q_h)
  && \text{for all } q_h \in Q_h.
\end{align}
\end{subequations}
The system \eqref{global:MCS} can be interpreted as an
$L^2$-minimization problem with constraints given by 
\begin{align*}
\| \sigma_h^\text{GEQ} - \nu \nabla
\bar{\vecb{u}}_h \|_{L^2(\Omega)} \rightarrow \min \quad \textrm{such that} \quad \langle
\divergence(\sigma_h), I_{\vecb{V}_h} \vecb{V}_0 \rangle_{\vecb{V}_h}
= (-\vecb{f}, I_{\vecb{V}_h}\vecb{V}_0).
\end{align*} 
Since no explicit basis for $I_{\vecb{V}_h}\vecb{V}_0$ (i.e. exactly
divergence-free Raviart-Thomas functions) can be constructed, the
divergence constraint is employed by means of the Lagrange multiplier
$p_h$. 

  \begin{theorem}\label{thm:MCSapriori} The solution
  \(\sigma_h^\text{GEQ}\) of problem \eqref{global:MCS} satisfies the
  discrete equilibration constraint \eqref{eq::relaxed_equi_disc}.
If the exact velocity solution of \ref{eq::stokes_ex} fulfills the regularity \(\vecb{u} \in
\vecb{H}^{m}(\mathcal{T}) \cap \vecb{V}_0\) with $m \geq 2$, it holds
the error estimate
  \begin{align*}
    \| \sigma - \sigma_h^\text{GEQ} \| \lesssim h^s \nu \| \vecb{u} \|_{H^{s+1}(\mathcal{T})}
  \end{align*}  
  where \(s := \min \lbrace m-1,k+1 \rbrace\) and \(\sigma = \nu \nabla \vecb{u}\).
\end{theorem}
\begin{proof}
  The equilibration constraint follows from the second equation of the discrete system \eqref{global:MCS}, since  given any \(\vecb{v} \in \vecb{V}_0\),
testing with the divergence-free function \(\vecb{v}_h := I_{\vecb{V}_h} \vecb{v}\) leads to
\(
  \langle \divergence(\sigma_h^\text{GEQ}), \vecb{v}_h \rangle_{\vecb{V}_h} = (-\vecb{f}, \vecb{v}_h)
\).

We continue with the error estimate by
showing that the solution of the best-approximation problem
\eqref{global:MCS} is related to solving a MCS-Stokes problem with a
zero right-hand in the first and third equation. To this end let
\(\vecb{w}_h =  \vecb{u}_h +  I_{\vecb{V}_h}
\bar{\vecb{u}}_h\). Since $\divergence(\tau_h) \in P_{k-1}(T)^2$ for
all $T \in \mathcal{T}$ and $\jump{(\tau_h)_{nn}} \in P_{k}(F)$ for
all $F \in \mathcal{F}$, the properties of the Raviart-Thomas
interpolator, integration by parts and the $H^1$-continuity of
$\bar{\vecb{u}}_h$  give
\begin{align*}
 \langle \divergence(\tau_h), I_{\vecb{V}_h} \bar{\vecb{u}}_h \rangle_{\vecb{V}_h} &=   \sum_{T \in \mathcal{T}} ( \divergence(\tau_h), I_{\vecb{V}_h} \bar{\vecb{u}}_h)_T
  - \sum_{F \in \mathcal{F}} (\jump{(\tau_h)_{nn}}, (I_{\vecb{V}_h} \bar{\vecb{u}}_h)_{n})_F \\
  &=   \sum_{T \in \mathcal{T}} (\divergence(\tau_h),  \bar{\vecb{u}}_h )_T
  - \sum_{F \in \mathcal{F}} ( \jump{(\tau_h)_{nn}} , (\bar{\vecb{u}}_h)_n )_F \\
  &=   - \sum_{T \in \mathcal{T}} ( \tau_h,  \nabla \bar{\vecb{u}}_h )_T
  + \sum_{F \in \mathcal{F}} ( (\tau_h)_{nt} ,\jump{  (\bar{\vecb{u}}_h)_t})_F \\
  &= -(\nabla \bar{\vecb{u}}_h, \tau_h).
\end{align*}
Further we have
$(\div(I_{\vecb{V}_h} \bar{\vecb{u}}_h), q_h) =
(\div(\bar{\vecb{u}}_h), q_h)$ for all $q_h \in Q_h$. This shows that
the triplet \((\sigma^\text{GEQ}_h,{\vecb{w}}_h,p_h) \in \Sigma_h
\times \vecb{V}_h \times Q_h\) solves the problem
\begin{align*}
  \frac{1}{\nu}(\sigma^\text{GEQ}_h,\tau_h) + \langle \divergence(\tau_h),{\vecb{w}}_h \rangle_{\vecb{V}_h} & = 0 
  && \text{for all } \tau_h \in \Sigma_h,\\
  \langle \divergence(\sigma^\text{GEQ}_h), \vecb{v}_h \rangle_{\vecb{V}_h} + (\divergence(\vecb{v}_h),p_h) &= (-\vecb{f},\vecb{v}_h)
  && \text{for all } \vecb{v}_h \in \vecb{V}_h,\\
  (\divergence({\vecb{w}}_h),q_h)  &=   0 
  && \text{for all } q_h \in Q_h.
\end{align*}
Since ${\vecb{w}}_h \in \vecb{V}_h$ is exactly divergence-free, the
pressure-robust error estimates of the standard Stokes problem
(discretised by the MCS method) from
\cite{lederer2019mass,Lederer:2019c} give 
\begin{align*}
 \| \sigma - \sigma^\text{GEQ}_h \|\lesssim h^s \nu \| \vecb{u} \|_{H^{s+1}}.
\end{align*}
This concludes the proof.
\end{proof}

The following theorem proves global efficiency of the global design. 

\begin{theorem}[Global efficiency of the global design] \label{th::globaleff}
  The error estimator for \(\sigma_h := \sigma_h^\text{GEQ}\) from
  \eqref{global:MCS}, is efficient in the sense that
  \begin{align*}
    \eta(\sigma_h ^\text{GEQ}) \lesssim \nu^{-1} \Big(\| \sigma - \bar \sigma_h \|
    + \| \sigma - \sigma_h^\text{GEQ} \| +  \mathrm{osc}_{k-2}(\curl(\vecb{f})) \Big),
  \end{align*}  
  with the data oscillations
  \begin{align} \label{eq::oscillations}
    \mathrm{osc}_{k-2}(\curl(\vecb{f})) := \Big(\sum_{T \in \mathcal{T}} h_T^2 \| (\id - \vecb{\pi}_{k-2})\curl (\vecb{f}) \|_{T}^2\Big)^{1/2}.
  \end{align}  
  \end{theorem}
  Note that the oscillations and the second term on the right-hand
  side (estimated in Theorem~\ref{thm:MCSapriori}) are of order
  $h_T^{k+1}$ if the data is smooth enough.  
  \begin{proof}
    This follows by the definition of the estimator $\eta$ of Theorem \ref{thm:EQestimator} and the
    triangle inequality 
  \begin{align*}
    \| \mathrm{dev}(\sigma_h -  \bar \sigma_h) \|
    \leq \| \sigma_h -  \bar \sigma_h \|
    \leq  \| \sigma - \bar \sigma_h \|
    + \| \sigma - \sigma_h \|,
  \end{align*}
  and the estimate
  \begin{align*}
    \| \mathrm{div}(\primaluh) \|
    = \| \mathrm{div}(u - \primaluh) \|
    \leq \nu^{-1} \| \sigma - \bar \sigma_h \|.
  \end{align*}
  
  \end{proof}


\section{Local Equilibration}\label{sec:localdesign}
This section suggests some design of an admissible pressure-robust equilibrated flux \(\sigma_h = \sigma^\text{LEQ}_h\)
for Theorem~\ref{thm:EQestimator} based on local problems on vertex patches.

\subsection{Setup of the local problems}
Let $\mathcal{V}$ be the set of vertices of a triangulation
$\mathcal{T}$. For $V \in \mathcal{V}$ let $\omega_V$ denote  the
corresponding vertex patch, i.e.\ the union of all adjacent cells in
\(\mathcal{T}_V := \lbrace T \in \mathcal{T} : V \in \overline{T}
\rbrace\). Furthermore, $\mathcal{F}_V$ denotes the set of facets
within the vertex patch including the facets on the boundary $\partial
\omega_V$.  For a fixed interior vertex $V$ we define the following
spaces with $k = r$ (recall that $r$ is the optimal convergence rate
of the primal method)
\begin{align*}
  \Sigma^V_h &:= \{ \tau_h \in L^2(\mathcal{T}_V)^{d \times d}: \forall T \in \mathcal{T}_V, \ \tau_h|_T \in P_k(T)^{d \times d} \text{ with } \mathrm{tr}(\tau_h) = 0\},\\ 
  \vecb{V}^V_h &:= \textrm{RT}_k(\mathcal{T}_V), \\
  \vecb{\hat V}^V_h &:= \{\vecb{\hat v}_h \in  \vecb{L}^2(\mathcal{F}_V): \forall F \in \mathcal{F}_V, \ \vecb{\hat v}_h|_{F} \in \vecb{P}_k(F) \textrm{ and } (\vecb{\hat v}_h)_n = 0\}, \\
  Q^V_h &:= \{ q_h \in L^2(\mathcal{T}_V): \forall T \in \mathcal{T}_V, \ q_h|_{T} \in P_k(T)\}.
\end{align*}
Note that in contrast to the global stress space $\Sigma_h$, the local
stress space $\Sigma^V_h$ is broken, i.e., does not include the  continuity constraint
$\jump{(\tau_h)_{nt}} = 0$. Similarly to other local equilibration
setups, see for example \cite{Braess:2008}, the space $\vecb{\hat
V}^V_h$ is chosen such that the normal-tangential trace of functions
in $\Sigma^V_h$ lie in $\vecb{\hat V}^V_h$. 

For the local problems we then further define the product space
\begin{align} \label{factorspace}
  \boldsymbol{X}^V_h :=  (\vecb{ V}^V_h \times \vecb{\hat V}^V_h) / \{ (\vecb{c}, \vecb{c}_t): \vecb{c} \in \R^d \},
\end{align}
where $\vecb{c}$ denotes a vector-valued constant, and $(\vecb{c},
\vecb{c}_t)$ reads as (a constant) element of the product space
$\vecb{ V}^V_h \times \vecb{\hat V}^V_h$. Hence, the space
$\boldsymbol{X}^V_h$ is factorised by vector-valued constant functions on
the patch.

The projection onto vector-valued constants $\vecb{\pi}_{\R}^V: \vecb{L}^2(\mathcal{T}_V)
\times [\vecb{L}^2(\mathcal{F}_V)]_t \rightarrow \R^d$ is given by
\begin{align*}
  \vecb{\pi}_{\R}^V(\vecb{v}_h,\vecb{\hat v}_h) := \frac{1}{|\mathcal{T}_V| + |\mathcal{F}_V|} \Big(\sum_{T \in \mathcal{T}_V} \int_T \vecb{ v}_h \dif x + \sum_{F \in \mathcal{F}_V} \int_F \vecb{\hat v}_h \dif s \Big) \in \R^d. 
\end{align*}
Here the quantities $|\mathcal{T}_V|$ and $|\mathcal{F}_V|$ denote the
area of the element patch and the skeleton of the patch respectively.
Note that we then have the equality
\begin{align} \label{factorspace::equal}
  \boldsymbol{X}^V_h = \{ (\vecb{ v}_h,\vecb{\hat v}_h) \in \vecb{V}^V_h \times \vecb{\hat V}^V_h: 
   (\id - \vecb{\pi}_{\R}^V)(\vecb{ v}_h,\vecb{\hat v}_h) \neq  \vecb{0} \}.
\end{align}
For each element $T \in \mathcal{T}$ and every vertex $V \in T$ we
define the scalar linear operator
\begin{align*}
  B_T^V: P_{k+1}(T) \rightarrow P_{k}(T), ~v_h \mapsto B_T^V(v_h) := I_{\mathcal{L}}^{k}(\phi_V v_h),
\end{align*}
where $I_{\mathcal{L}}^{k}$ denotes the nodal (Lagrange)
interpolation operator into $P_{k}(T)$ and $\phi_V$ is the hat
function of the vertex $V$. By that we then define on $\omega_V$ the
scalar bubble projector (see also \cite{2016arXiv160903701L})
\begin{align*}
B^V: P_{k+1}(\mathcal{T}_V) \rightarrow P_{k}(\mathcal{T}_V), ~v_h \mapsto B^V(v_h) := \sum\limits_{T \in \mathcal{T}_V}B_T^V(v_h),
\end{align*}
and the (vector-valued) bubble projector
\begin{align*}
  \boldsymbol{B}^V: \boldsymbol{P}_{k+1}(\mathcal{T}_V) \rightarrow \boldsymbol{P}_{k}(\mathcal{T}_V) , ~\vecb{v}_h \mapsto \boldsymbol{B}^V(\vecb{v}_h)\quad \textrm{with} \quad (\boldsymbol{B}^V(\vecb{v}_h))_j:= B^V((\vecb{v}_h)_j),
\end{align*}
where $(\cdot)_j$ denotes the $j$-th component of the vector.
\begin{lemma} \label{lem::bubbleprops}
  The bubble projector $\boldsymbol{B}^V$ fulfills the following properties:
  \begin{enumerate} 
  \item \label{lem::bubbleprops_one}
  $\boldsymbol{B}^V(\vecb{v}_h)|_{\partial \omega_V} =0$ for
  all $\vecb{v}_h \in \vecb{V}_h^V$ 
  \item \label{lem::bubbleprops_two} Let $\vecb{ v}_h^V \in \vecb{
  V}_h^V$ then $\boldsymbol{B}^V(\vecb{ v}_h^V) \in
  \mathrm{BDM}_k(\mathcal{T}_V)$.  
  \item \label{lem::bubbleprops_three} For all elements $T \in \mathcal{T}$ we have the partition of unity property
    \begin{align*}
    \sum\limits_{V \in T} \boldsymbol{B}^V(\vecb{v}_h|_T) =  \vecb{v}_h|_T  \quad \text{for all } \vecb{v}_h \in \vecb{V}_h \textrm{ with } \div(\vecb{v}_h) = 0.
    \end{align*}
\item \label{lem::bubbleprops_four} For a constant $\vecb{c} \in \R^d$ there holds $\boldsymbol{B}^V(\vecb{c}) = \phi_V \vecb{c}$.
  \end{enumerate}
\end{lemma}
\begin{proof}
  Item \ref{lem::bubbleprops_one} follows by the definition and the
  linearity of the bubble projection. For the proof of
  \ref{lem::bubbleprops_two} choose an arbitrary edge $F \in
  \mathcal{F}_V$ with the corresponding normal vector
  $\boldsymbol{n}$. Since $\vecb{v}_h^V$ is normal continuous we have
  by the properties of the nodal interpolation operator that
  \begin{align*}
\jump{(\boldsymbol{B}^V(\vecb{ v}_h^V)(\vecb{x}))_n} = \phi_V(\vecb{x})\jump{(\vecb{ v}_h^V(\vecb{x}))_n}=0 \quad \text{for all } \vecb{x} \in F.
  \end{align*}
  This shows that $\boldsymbol{B}^V(\vecb{ v}_h^V)$ is normal
  continuous and as $\boldsymbol{B}^V(\vecb{ v}_h^V) \in
  \vecb{P}_k(\mathcal{T}_V)$ it follows that
  $\boldsymbol{B}^V(\vecb{v}_h^V) \in \mathrm{BDM}_k(\mathcal{T}_V)$.
  For \ref{lem::bubbleprops_three} let $\vecb{v}_h \in \vecb{V}_h$
  with $\div(\vecb{v}_h) = 0$ be arbitrary. Since $\vecb{v}_h$ is
  divergence-free it follows that $\vecb{v}_h \in
  \mathrm{BDM}_k(\mathcal{T})$, thus $\vecb{v}_h|_T \in \vecb{P}_k(T)$
  and so $I^{k}_{\mathcal{L}}(\vecb{v}_h|_T) = \vecb{v}_h|_T$. The
  claimed partition of unity property then follows by the linearity of
  the bubble projector $\boldsymbol{B}^V$.
  The last item \ref{lem::bubbleprops_four} follows for each
  component separately since there holds $I^{k}_{\mathcal{L}}(c
  \phi_V) = c \phi_V$ on each element
  $T\in\mathcal{T}_V$ and all $c \in \R$. This concludes the proof.
\end{proof}
For each vertex $V$ we solve the local problem: Find $(\sigma^V_h,
(\vecb{ u}^V_h, \vecb{\hat u}^V_h), p^V_h) \in \Sigma^V_h \times
\boldsymbol{X}^V_h \times Q^V_h$ such that, for all $(\tau^V_h,
(\vecb{ v}^V_h, \vecb{\hat v}^V_h), q^V_h) \in \Sigma^V_h \times
\boldsymbol{X}^V_h \times Q^V_h$,
\begin{subequations}
  \label{eq:local_prob}
  \begin{align}
     (\sigma^V_h, \tau^V_h)_{\omega_V} + \langle \div(\tau^V_h),(\vecb{ u}^V_h, \vecb{\hat u}^V_h)\rangle_{\vecb{V}_h^V}&= 0\\
     \langle \div (\sigma^V_h), (\vecb{ v}^V_h, \vecb{\hat v}_h^V)\rangle_{\vecb{V}_h^V} + (\divergence(\vecb{ v}^V_h), p^V_h)_{\omega_V} &= \vecb{r}^V((\vecb{ v}^V_h,\vecb{\hat v}_h^V)), \label{eq:local_probtwo}\\
     (\divergence(\vecb{ u}^V_h), q^V_h)_{\omega_V} &= 0,
  \end{align}
\end{subequations}
with the bilinear form
\begin{subequations}
  \begin{align*}
    \langle \div (\tau^V_h),(\vecb{ u}^V_h, \vecb{\hat u}^V_h)\rangle_{\vecb{V}_h^V}
    :=& \sum_{T\in\mathcal{T}_V}(\divergence(\tau_h^V) , \vecb{ u}^V_h)_T\\
        &- \sum_{F \in \mathcal{F}_V}  ( \jump{(\tau_h^V)_{nn}}, (\vecb{ u}^V_h)_n)_F + (\jump{(\tau_h^V)_{nt}}, \vecb{\hat u}_h^V)_F ,
  \end{align*}
  \label{eq:bilin_forms_local}
\end{subequations}
and the local residuum for the solution $\primaluh,\primalph$ of
\eqref{primalform} with $\bar \sigma_h = \nu \nabla \primaluh$
given by 
\begin{multline*}
   \vecb{r}^V((\vecb{ v}^V_h, \vecb{\hat v}_h^V)) := \sum_{T\in\mathcal{T}_V} ( \vecb{f} , \boldsymbol{B}^V(\vecb{ v}_h^V) )_T + \sum_{T\in\mathcal{T}_V} ( \nu \Delta \primaluh - \nabla \primalph,  \boldsymbol{B}^V(\vecb{ v}_h^V))_T\\
  - ( (\bar \sigma_h - \primalph \identity)_{nn}, (\boldsymbol{B}^V(\vecb{ v}_h^V))_n)_{\partial T}- ( \phi_V (\bar \sigma_h)_{nt},\vecb{\hat v}_h^V)_{\partial T}.
\end{multline*}
Note that $\langle \div (\cdot), (\cdot, \cdot)\rangle_{\vecb{V}_h^V}$
reads as a restriction of the discrete duality pair $\langle
\divergence(\cdot), \cdot \rangle_{\vecb{V}_h}$ onto $\omega_V$ but
further includes the normal-tangential jumps since functions in
$\Sigma^V_h$ are not (normal-tangential) continuous. This shows that
\eqref{eq:local_prob} reads as a local version of the global
problem given by \eqref{global:MCS} where the (normal-tangential)
continuity of the stress variable $\sigma_h^V$ is incorporated by a
Lagrange multiplier in $\vecb{ \hat V}_h^V$. 

Using integration by parts, the right hand side can also be
written as
\begin{align} \label{rhsintbyparts}
    \vecb{r}^V((\vecb{ v}^V_h, \vecb{\hat v}_h^V)) =  &\sum_{T\in\mathcal{T}_V}( \vecb{f}, \boldsymbol{B}^V(\vecb{ v}^V_h) )_T - ( \bar \sigma_h, \nabla \boldsymbol{B}^V(\vecb{ v}^V_h))_T \\
    &+ (\primalph, \divergence(\boldsymbol{B}^V(\vecb{ v}^V_h)))_T 
                                + (  (\bar \sigma_h)_{nt}, ( \boldsymbol{B}^V(\vecb{ v}_h^V) - \phi_V \vecb{\hat v}_h^V)_t)_{\partial T}. \nonumber
\end{align}

\begin{remark}
  As usual for equilibrated error estimators we slightly modify the
  definition of the local problems when the vertex $V$ lies on the
  Dirichlet boundary. In this case the degrees of freedom
  of $\vecb{\hat V}^V_h$ on the domain boundary are removed. Accordingly, $\mathcal{F}_V$
  gets replaced by $\mathcal{F}_V \setminus \{F \in
  \mathcal{F}_V: F \subset \partial \Omega \}$. Moreover, the
  mean value constraint of the product space is removed, i.e.\ we set
  $\boldsymbol{X}^V_h :=  (\vecb{ V}^V_h \times \vecb{\hat
  V}^V_h)$.
\end{remark}
\subsection{Analysis of the local problem}
For the analysis consider the norms
\begin{align*}
  \| \sigma_h^V \|_{\Sigma^V_h}^2 &:= \sum_T \| \sigma_h^V \|^2_{T} + h_T \| (\sigma^V_h)_{nt}\|^2_{\partial T}, \\
  \| (\vecb{ v}^V_h,\vecb{\hat v}^V_h) \|_{\vecb{X}^V_h}^2 &:= \sum_T \| \nabla \vecb{ v}_h^V \|^2_{T} + \frac{1}{h_T} \| (\vecb{\hat v}_h^V - \vecb{ v}_h^V)_t \|^2_{\partial T},\\
  \| p_h^V \|_{Q^V_h} &= \| p_h^V \|_{\omega_V}.
\end{align*}
Note that the the norm $ \| \cdot \|_{\vecb{X}^V_h}$ reads as an
$H^1$-like norm on the velocity space $\vecb{X}_h$ since the Lagrange
multipliers in $\vecb{\hat V}^V_h$ can be interpreted as the
tangential component of the local velocities in $\vecb{V}^V_h$. Such
norms are very common in the analysis of hybrid discontinuous Galerkin
methods, see for example \cite{doi:10.1137/17M1138078}. Further we define the kernel of the constraints given by
\begin{align*} \mathcal{K}^V := \{(\sigma_h^V,p_h^V) &\in \Sigma_h^V
\times Q_h^V: \forall (\vecb{v}_h^V, \vecb{\hat v}_h^V) \in
\boldsymbol{X}^V_h, \\
&\langle \div (\sigma^V_h), (\vecb{ v}^V_h,
\vecb{\hat v}_h^V)\rangle_{\vecb{V}_h^V} + (\divergence(\vecb{v}^V_h),
p^V_h)_{\omega_V} = 0\}.
\end{align*}
\begin{lemma}\label{lemma::stability} 
  The following stability conditions hold true:
  \begin{itemize}
    \item Continuity: For all  $\sigma_h^V, \tau_h^V \in
    \Sigma_h^V$, $(\vecb{ v}_h^V, \vecb{\hat
    v}_h^V) \in \boldsymbol{X}^V_h$ and $q_h \in Q_h^V$ we have
    \begin{align*}
      (\sigma_h^V, \tau_h^V)_{\omega_V} &\lesssim \| \sigma_h^V
      \|_{\Sigma^V_h} \| \tau^V \|_{\Sigma^V_h}, \\
      \langle \div (\sigma^V_h), (\vecb{ v}^V_h, \vecb{\hat
v}_h^V)\rangle_{\vecb{V}_h^V} & \lesssim \| \sigma_h^V \|_{\Sigma^V_h}
\| (\vecb{ v}^V_h,\vecb{\hat v}^V_h) \|_{\vecb{X}^V_h}, \\
(\divergence(\vecb{v}^V_h), q^V_h)_{\omega_V} &\lesssim \| (\vecb{
v}^V_h,\vecb{\hat v}^V_h) \|_{\vecb{X}^V_h} \| q_h\|_{Q_h^V}.
    \end{align*}
  \item Kernel coercivity: For all $(\sigma_h^V,p_h^V)\in
  \mathcal{K}^V$ we have
  \begin{align*}
    (\sigma^V_h, \tau^V_h)_{\omega_V} \gtrsim (\| \sigma_h^V \|_{\Sigma^V_h}^2 + \| p_h^V \|_{Q^V_h}^2).
  \end{align*}
  \item The inf-sup conditions:
  \begin{enumerate}
  \item \label{infsup_one} For all $(\vecb{ v}_h^V, \vecb{\hat v}_h^V) \in
  \boldsymbol{X}^V_h$ there exists a constant $\beta_1 > 0$ such that
    \begin{align*}
      \sup\limits_{(\sigma_h^V,p_h^V) \in \Sigma_h^V \times Q_h^V} \frac{\langle \div (\sigma^V_h), (\vecb{ v}^V_h,
      \vecb{\hat v}_h^V)\rangle_{\vecb{V}_h^V} + (\div(\vecb{ v}^V_h), p^V_h)_{\omega_V}}{\| \sigma_h^V \|_{\Sigma^V_h}\| + \| p_h^V \|_{Q^V_h}} \ge \beta_1 \| (\vecb{ v}^V_h,\vecb{\hat v}^V_h) \|_{\vecb{X}^V_h}.
    \end{align*}
    \item \label{infsup_two} For all $(\vecb{ v}_h^V, \vecb{\hat v}_h^V) \in \boldsymbol{X}^V_h$ with $\divergence(\vecb{ v}^V_h) = 0$ there exists a constant $\beta_2 > 0$ such that
    \begin{align*}
      \sup\limits_{\sigma_h^V \in \Sigma_h^V} \frac{\langle \div (\sigma^V_h), (\vecb{ v}^V_h,
      \vecb{\hat v}_h^V)\rangle_{\vecb{V}_h^V}}{\| \sigma_h^V \|_{\Sigma^V_h} } \ge \beta_2 \| (\vecb{ v}^V_h,\vecb{\hat v}^V_h) \|_{\vecb{X}^V_h}.
    \end{align*}
  \end{enumerate}
\end{itemize}
\end{lemma}
\begin{proof}
  The continuity follows immediately with the Cauchy--Schwarz
  inequality and using integration by parts for integrals of the
  bilinear form  $\langle \div (\cdot), (\cdot,
  \cdot)\rangle_{\vecb{V}_h^V} $. The proofs of the kernel ellipticity
  and the inf-sup conditions follow with exactly the same steps as in
  the stability proofs of the original MCS-method in
  \cite{10.1093/imanum/drz022,Lederer:2019c,lederer2019mass}, since
  the bilinear forms and spaces of the local problems in this work
  simply read as a hybridized version of the original MCS-method. In
  this work the normal-tangential continuity of the stress space is
  incorporated by the additional Lagrange multiplier $\vecb{\hat
  u}_h^V$ and we switched from the $H^1$-like DG norm used in the
  original works to the corresponding $H^1$-like HDG norm given by $\|
  \cdot \|_{\vecb{X}^V_h}$ in this work. Note however, that we do not
  have zero Dirichlet boundary conditions of the velocity variable,
  but since we excluded the kernel of $\| \cdot \|_{\vecb{X}^V_h}$
  (constant functions) in the definition of the space
  $\boldsymbol{X}^V_h$, the results simply follow by norm equivalence.
\end{proof}
\begin{theorem}\label{theorem::stability} There exists a unique
  solution $(\sigma_h^V,(\vecb{ u}_h^V,\vecb{\hat u}_h^V), p_h^V) \in
  \Sigma^V_h \times \boldsymbol{X}^V_h \times Q^V_h$ of
  \eqref{eq:local_prob} with the stability estimate
  \begin{align*}
    \| \sigma_h^V \|_{\Sigma^V_h} + \| (\vecb{u}^V_h,\vecb{\hat u}^V_h) \|_{\vecb{X}^V_h} &\lesssim \sup\limits_{\substack{(\vecb{ v}_h^V, \vecb{\hat v}_h^V) \in \boldsymbol{X}^V_h\\\div(\vecb{ v}_h^V) = 0}} \frac{\vecb{r}^V((\vecb{ v}^V_h,\vecb{\hat v}^V_h))}{\| (\vecb{ v}^V_h,\vecb{\hat v}^V_h) \|_{\vecb{X}^V_h}}
  \end{align*}
\end{theorem}
\begin{proof}
  The solvability of \eqref{eq:local_prob} follows with the standard
  theory of saddle point problems (i.e.\ Brezzi's Theorem), see for
  example in \cite{Boffi2008Mixed} and the estimitates of Lemma
  \ref{lemma::stability}. The stability estimate follows by the
  inf-sup condition on the subspace of divergence-free functions, see
  item \ref{infsup_two} of Lemma \ref{lemma::stability}, and standard
  estimates employing the solvability of \eqref{eq:local_prob}.
\end{proof}

Now let $ \vecb{\hat V}_h$ be the global version of the local space $
\vecb{\hat V}^V_h$, i.e. we define 
\begin{align*}
   \vecb{\hat V}_h &:= \{\vecb{\hat v}_h \in  \vecb{L}^2(\mathcal{F}): \forall F \in \mathcal{F}, \ \vecb{\hat v}_h|_{F} \in \vecb{P}_k(F) \textrm{ and } (\vecb{\hat v}_h)_n = 0\}.
\end{align*}
The next theorem provides several properties of the local solution of
\eqref{eq:local_prob}. First we show that equation \eqref{eq:local_probtwo}
also holds for constants (and not only for functions of the factor
space $\vecb{X}_h^V$), and that the local stress variable has a
vanishing normal-tangential trace. Further we discuss local
equilibrium conditions.
\begin{theorem}[Properties of the local solution]
   \label{localproperties} Let $\sigma_h^V \in \Sigma^V_h$ and $p_h^V
   \in Q_h^V$ be the local solution of \eqref{eq:local_prob}.
   There hold the following properties:
  \begin{enumerate}
    \item For all $(\vecb{v}^V_h,\vecb{\hat v}_h^V) \in \vecb{V}_h^V
    \times \vecb{\hat V}_h^V$ there holds
    \begin{align} \label{equilibrationlocal}
      \langle \div(\sigma^V_h), (\vecb{v}^V_h, \vecb{\hat v}_h^V)\rangle_{\vecb{V}_h^V} + (\div(\vecb{ v}^V_h), p^V_h)_{\omega_V} &= \vecb{r}^V((\vecb{ v}^V_h,\vecb{\hat v}_h^V)).
    \end{align}
  \item For any $(\vecb{v}_h,\vecb{\hat v}_h)$ of the (global) space
  $\vecb{V}_h \times \vecb{\hat V}_h$, with $\divergence( \vecb{v}_h)
  = 0$, there holds the local equilibrium condition
  \begin{align}\label{equilibrationgloballocal}
    \langle \div(\sigma^V_h), (\vecb{v}_h, \vecb{\hat v}_h)\rangle_{\vecb{V}^V_h} &= \vecb{r}^V((\vecb{ v}_h,\vecb{\hat v}_h)).
  \end{align}
  \item The solution $\sigma^V_h$ has a zero normal-tangential trace at the boundary
    \begin{align*}
    (\sigma_h^V)_{nt}=0 \quad \textrm{on} \quad \partial \omega_V.  
    \end{align*}    
      \end{enumerate}
\end{theorem}
\begin{proof}
  Let $V \in \mathcal{V}$ be an arbitrary but fixed vertex. In a first
  step we will proof that equation \eqref{eq:local_probtwo} also hold
  for constant functions. To this end let $\vecb{c} =
  \vecb{\pi}_\R^V((\vecb{ v}_h^V,\vecb{\hat v}_h^V))$. Using
  $\divergence(\vecb{c}) = 0$ and integration by parts we have for the
  left side of
  \eqref{eq:local_probtwo}
  \begin{align*}
    \langle \div(\sigma^V_h), (\vecb{c}, \vecb{c}_t)\rangle_{\vecb{V}_h^V} &+ (\div(\vecb{c}), p^V_h)_{\omega_V}\\
                           =&\sum_T ( \divergence(\sigma^V_h) , \vecb{c} )_T - \sum_F ( \jump{(\sigma^V_h)_{nn}}, \vecb{c}_n)_F +  (\jump{(\sigma^V_h)_{nt}}, \vecb{c}_t)_F \\
                           =&-\sum_T (\sigma^V_h, \nabla \vecb{c})_T  + \sum_F ( \jump{(\sigma^V_h)_{nt}} , (\vecb{c} - \vecb{c})_t)_F = 0.
  \end{align*}
  We continue with the right-hand side. Using representation
  \eqref{rhsintbyparts} we get for the constant $\vecb{c}$ and using
  item \ref{lem::bubbleprops_four} of Lemma~\ref{lem::bubbleprops} that
  \begin{align*}
    \vecb{r}^V((\vecb{c},\vecb{c}_t))
    =  &\sum_{T\in\mathcal{T}_V} (\vecb{f}, \boldsymbol{B}^V(\vecb{c}))_T - ( \bar \sigma_h,  \nabla\boldsymbol{B}^V(\vecb{c}))_T +( \phi_V \primalph, \divergence(\boldsymbol{B}^V(\vecb{c})))_T\\
         & + ((\bar \sigma_h)_{nt},  (\boldsymbol{B}^V(\vecb{c})_t - \phi_V \vecb{c}_t) )_{\partial T} \\
    = & \sum_{T\in\mathcal{T}_V}(\vecb{f} , \vecb{c} \phi_V)_T - ( \bar \sigma_h ,\nabla  (\vecb{c} \phi_V))_T +( \primalph\divergence( \vecb{c} \phi_V))_T.
  \end{align*}
  Since $\vecb{c} \phi_V$ is an element of the velocity Stokes
  discretisation space $\vecb{\bar V}_h$, the first line of
  \eqref{primalform} then gives
  \begin{align*}
    \vecb{r}^V&((\vecb{c},\vecb{c}_t)) =  \sum_{T\in\mathcal{T}_V} ( \vecb{f}, (\id - \mathcal{R})\vecb{c} \phi_V)_T = 0,
  \end{align*}
  where the last step follows from $\vecb{c} \phi_V \in
  \vecb{P}_{1}(\mathcal{T}) \cap \vecb{H}^1_0(\Omega)$ and
  \eqref{rec_three}. In total this shows that we also have $\vecb{r}^V
  ((\vecb{c},\vecb{c}_t)) = 0$, thus using $(\vecb{v}_h^V, \vecb{\hat
  v}_h^V) = (\vecb{\pi}_\R^V + (\id - \vecb{\pi}_\R^V))(\vecb{ v}_h^V,
  \vecb{\hat v}_h^V)$, definition \eqref{factorspace::equal} and
  \eqref{eq:local_probtwo}, we have proven \eqref{equilibrationlocal}. 
  
  For the second statement let $(\vecb{v}_h,\vecb{\hat v}_h) \in
  \vecb{V}_h \times \vecb{\hat V}_h$, with $\divergence(\vecb{v}_h) =
  0$. Setting $\vecb{v}_h^V = \vecb{v}_h|_{\omega_V}$ and $\vecb{\hat
  v}_h^V = \vecb{\hat v}_h|_{\mathcal{F}_V}$ immediately proves \eqref{equilibrationgloballocal}.
  
  For the proof of the third statement consider the test function
  $\vecb{\hat v}_h^V \in \vecb{\hat V}_h^V$ such that $\vecb{\hat
  v}_h^V = (\sigma_h^V)_{nt}$ on every facet $F \subset \partial
  \omega_V$, and zero on the internal facets. Using equation
  \eqref{equilibrationlocal} with $\vecb{v}_h^V = 0$  then gives (for
  the squared $L^2$-norm on the facets)
  \begin{align*}
    \sum_{F \in \partial \omega_V} ((\sigma_h^V)_{nt},(\sigma_h^V)_{nt})_F
    = \sum_{F \in \partial \omega_V} ( (\sigma_h^V)_{nt} , \vecb{\hat v}_h^V)_F 
    = \sum_{T \in \mathcal{T}_V} ( \phi_V (\bar \sigma_h)_{nt} , \vecb{\hat v}_h^V)_{\partial T} = 0,
  \end{align*}
  where we used that $\phi_V$ vanishes on the boundary $\partial \omega_V$ and $\vecb{\hat v}_h^V $ on internal facets.
\end{proof}

\subsection{Admissibility of the global flux}
After solving the local problems we define the equilibrated flux 
\begin{align}\label{localeqflux}
\sigma_h^\text{LEQ} := \bar \sigma_h - \sigma_h^\Delta \quad \text{with} \quad \sigma_h^\Delta := \sum_{V} \sigma_h^V.
\end{align}
\begin{theorem}
  The locally equilibrated stress $\sigma_h^\text{LEQ}$ is an element
  of $\Sigma_h$ and satisfies the discrete equilibration condition \eqref{eq::relaxed_equi_disc}, i.e.,
\begin{align*}
  (\vecb{f},\vecb{v}_h) + \langle \divergence(\sigma^\text{LEQ}_h), \vecb{v}_h \rangle_{\vecb{V}_h} = 0 \quad \text{for all } \vecb{v}_h \in \vecb{V}_h \textrm{ with } \div(\vecb{v}_h) = 0.
\end{align*}
\end{theorem}
\begin{proof} To show $\sigma_h^\text{LEQ} \in
\Sigma_h$ it suffices to show that its normal-tangential jumps vanish.
Indeed, for any arbitrary $\vecb{\hat v}_h \in \vecb{ \hat V}_h$, it holds,
  \begin{align*}
    \sum_{F \in \mathcal{F}} ( \jump{(\sigma_h^\text{LEQ})_{nt}} , (\vecb{\hat v}_h)_t)_F 
    &= \sum\limits_{V \in \mathcal{V}} \sum\limits_{F \in \mathcal{F}_V} ( \jump{(\sigma_h^\text{LEQ})_{nt}} , ( \phi_V \vecb{\hat v}_h)_t )_F \\
    &= \sum\limits_{V \in \mathcal{V}} \sum\limits_{F \in \mathcal{F}_V}  ( \jump{(\bar \sigma_h)_{nt}}, (\phi_V \vecb{\hat v}_h)_t)_F-  (\jump{(\sigma_h^\Delta)_{nt}}, (\phi_V\vecb{\hat v}_h)_t)_F.
  \end{align*}
  Here, we only used that the $\lbrace \phi_V \rbrace_{V \in \mathcal{V}}$ form a partition of unity. Applying the third and then
  the second statement of Theorem~\ref{localproperties} (with
  $\vecb{v}_h = 0$), the sum over the last integral can be written as
  \begin{align*}
    \sum\limits_{V \in \mathcal{V}} \sum\limits_{F \in \mathcal{F}_V}( \jump{(\sigma_h^\Delta)_{nt}}, (\phi_V\vecb{\hat v}_h)_t)_F &=  \sum\limits_{V \in \mathcal{V}} \sum\limits_{F \in \mathcal{F}_V}  ( \jump{(\sigma_h^V)_{nt}}, (\phi_V\vecb{\hat v}_h)_t)_F \\
    &=  \sum\limits_{V \in \mathcal{V}} \sum\limits_{F \in \mathcal{F}_V}  ( \jump{(\bar \sigma_h)_{nt}}, (\phi_V \vecb{\hat v}_h)_t )_F,
  \end{align*}
  and thus $ \sum_{F \in \mathcal{F}} (
    \jump{(\sigma_h^\text{LEQ})_{nt}} ,(\vecb{\hat v}_h)_t)_F
    =0$.
  With the choice $(\vecb{\hat v}_h)_t = \jump{(\sigma_h^\text{LEQ})_{nt}}$, we conclude that $\jump{(\sigma_h^\text{LEQ})_{nt}} = 0$ point wise, and so $\sigma_h^\text{LEQ} \in \Sigma_h$.

To show the equilibration constraint, consider an arbitrary
  $\vecb{v}_h \in \vecb{V}_h$ with $\divergence(\vecb{v}_h) = 0$.
  Since $\sigma_h^\text{LEQ} \in \Sigma_h$, the definition of the
  global and local distributional divergences and \eqref{localeqflux}
  gives
\begin{align*}
  \langle \divergence(\sigma^\text{LEQ}_h), \vecb{v}_h \rangle_{\vecb{V}_h} 
  &= \langle \divergence(\bar \sigma_h), \vecb{v}_h \rangle_{\vecb{V}_h}  - \sum\limits_{V\in \mathcal{V}} \langle \divergence( \sigma^V_h), (\vecb{v}_h,\vecb{ 0}) \rangle_{\vecb{V}^V_h} \\
  & = \langle \divergence(\bar \sigma_h), \vecb{v}_h \rangle_{\vecb{V}_h} - \sum_{V \in \mathcal{V}} \vecb{r}^V((\vecb{ v}_h,\vecb{0}))  = -(\vecb{f}, \vecb{v}_h ).
\end{align*}
The last identity follows with an integration by parts, item
        \ref{lem::bubbleprops_three} of Lemma \ref{lem::bubbleprops}
        and $\divergence(\vecb{v}_h) = 0$. Altogether this shows the
        claimed discrete equilibration condition.
\end{proof}

\section{Local Efficiency}\label{sec:efficiency}

This section proves efficiency of the proposed local equilibrated
fluxes in the sense that the error estimator is a lower bound for the
velocity error plus norms that only depend on the velocity and have
the right order and data oscillations. In particular also the
efficiency bound is pressure-independent.

\begin{theorem}[Efficiency of the local
  design]\label{th::localprop}  
  Assume that the exact solution fulfills the regularity $ \vecb{u}
  \in \vecb{H}^m(\mathcal{T}) \cap V_0$, for some $m \ge 2$. The error
  estimator for \(\sigma_h := \sigma_h^\text{LEQ}\) is efficient in
  the sense that
  \begin{align*}
    \eta(\sigma_h ^\text{LEQ}) \lesssim \nu^{-1} \Big(\| \sigma - \bar \sigma_h \|
     + \sum\limits_{V \in \mathcal{V} } \| \sigma_ h^V \|_{\Sigma_h^V} +  \mathrm{osc}_{k-2}(\curl(\vecb{f})) \Big),
  \end{align*} 
  where $\sigma_h^V \in \Sigma_h^V$ are the local solutions of
\eqref{eq:local_prob} and the oscillations as in
Theorem~\ref{th::globaleff}. Further there holds for all local
solutions the pressure-robust local efficiency
  \begin{align*}
    \| \sigma_ h^V \|^2_{\Sigma^V_h}  &\lesssim \sum_{T \in \mathcal{T}_V}\|\sigma -\bar \sigma_h\|^2_{T} + h_T \| (\sigma - \bar \sigma_h)_{nt}\|^2_{\partial T} + h_T^2\|(\id - \vecb{\pi}_{r-2}) \nu\Delta \vecb{u}\|^2_{T}.
  \end{align*}
 If the reconstruction operator $\mathcal{R}$ of
the primal method \eqref{primalform} is the identity, the last term of
the right hand side vanishes.
\end{theorem}
\begin{proof}
  The first statement follows with exactly the same steps as in the
  proof Theorem~\ref{th::globaleff}, equation \eqref{localeqflux} and the
  triangle inequality 
  \begin{align*}
    \| \sigma - \sigma_h^\text{LEQ} \| \le \| \sigma - \bar \sigma_h \| + \| \sum\limits_{V \in \mathcal{V} } \sigma_ h^V \| \le \| \sigma - \bar \sigma_h \| +  \sum\limits_{V \in \mathcal{V} } \| \sigma_ h^V \|_{\Sigma_h^V}.
  \end{align*}
  We continue with the proof of the local pressure-robust stability
  estimate. For this let $(\sigma_h^V,(\vecb{u}_h^V,\vecb{\hat
  u}_h^V), p_h^V) \in \Sigma^V_h \times \boldsymbol{X}^V_h \times
  Q^V_h$ be the solution of \eqref{eq:local_prob}. By the stability
  estimate of Theorem \ref{theorem::stability} we have
\begin{align*}
      \| \sigma_ h^V \|_{\Sigma^V_h} &\lesssim \sup\limits_{\substack{(\vecb{v}_h^V,\vecb{\hat v}_h^V) \in \boldsymbol{X}_h^V \\ \divergence(\vecb{ v}^V_h) = 0 }} \frac{ \vecb{r}^V((\vecb{ v}^V_h,\vecb{\hat v}_h^V)) }{ \| (\vecb{ v}^V_h,\vecb{\hat v}^V_h) \|_{\vecb{X}^V_h}}.
    \end{align*}
    Let $(\vecb{v}_h^V,\vecb{\hat v}_h^V) \in \boldsymbol{X}_h^V$ with
     $\divergence (\vecb{v}_h) = 0$ be arbitrary. With $\vecb{f} = -
     \Delta \vecb{u} + \nabla p$ and applying integration by parts
     (similar to \eqref{rhsintbyparts}), the numerator simplifies to
  \begin{align}
    \vecb{r}^V((\vecb{ v}^V_h,\vecb{\hat v}_h^V))
    &= \sum_{T\in\mathcal{T}_V} (\sigma - \bar \sigma_h, \nabla \boldsymbol{B}^V(\vecb{ v}^V_h))_T \label{eqn:locefficiency_proof_firstterm}\\
    &-\sum_{T\in\mathcal{T}_V}  ((\sigma - \bar \sigma_h)_{nt}, (\boldsymbol{B}^V(\vecb{ v}_h^V) -\phi_V \vecb{\hat v}_h^V)_t)_{\partial T} \label{eqn:locefficiency_proof_secondterm} \\
                               &-\sum_{T\in\mathcal{T}_V}  (p - \primalph, \divergence (\boldsymbol{B}^V(\vecb{ v}^V_h)))_T,\label{eqn:locefficiency_proof_thirdterm}
  \end{align}
  where we used that $\boldsymbol{B}^V(\vecb{ v}^V_h)\cdot
  \vecb{n} = 0$ on $\partial \omega_V$ (see item
  \ref{lem::bubbleprops_one} of Lemma \ref{lem::bubbleprops}) and that
  \begin{align*}
    \sum_{T\in\mathcal{T}_V} (\phi_V (\sigma)_{nt}, (\vecb{\hat v}_h^V)_t)_{\partial T}  =     \sum_{F \in\mathcal{F}_V} (\jump{\phi_V(\sigma)_{nt}}, (\vecb{\hat v}_h^V)_t)_F  = 0.
  \end{align*}
  By the continuity of the bubble projector $\boldsymbol{B}^V$ and
  that $\phi_V = \mathcal{O}(1)$ on $\omega_V$, the Cauchy--Schwarz
  inequality applied to the sums in
  \eqref{eqn:locefficiency_proof_firstterm} and
  \eqref{eqn:locefficiency_proof_secondterm} gives
  \begin{align*}
    &\sum_{T\in\mathcal{T}_V} (\sigma - \bar \sigma_h, \nabla \boldsymbol{B}^V(\vecb{ v}^V_h))_T -((\sigma - \bar \sigma_h)_{nt}, (\boldsymbol{B}^V(\vecb{ v}_h^V) - \phi_V\vecb{\hat v}_h^V)_t)_{\partial T}\\
    &\le \sum_{T\in\mathcal{T}_V} \|(\sigma - \bar \sigma_h)\|_T \| \nabla \vecb{ v}^V_h\|_T + h_T \| (\sigma - \bar \sigma_h)_{nt} \|_{\partial T} \frac{1}{h_T} \| (\vecb{ v}_h^V -\vecb{\hat v}_h^V)_t \|_{\partial T} \\
    &\le \del{\sum_{T\in\mathcal{T}_V} \|(\sigma - \bar \sigma_h)\|^2_T + h_T^2 \| (\sigma - \bar \sigma_h)_{nt} \|^2_{\partial T}}^{1/2} \| (\vecb{ v}_h^V,\vecb{\hat v}_h^V)\|_{\vecb{X}_h^V}.
  \end{align*}
 We continue with the remaining third sum in
 \eqref{eqn:locefficiency_proof_thirdterm} (which does not vanish,
 although \(\vecb{ v}_h^V\) is divergence-free). For this let
 $\pi^{\primalQh} p$ be the $L^2$ best-approximation of the exact pressure
 in the pressure space $\primalQh$ and define the mean value
  \begin{align*}
    c_p := \frac{1}{|\mathcal{T}_V|} (\pi^{\primalQh}p - \primalph, 1)_{\omega_V}.
  \end{align*}
  According to item
  \ref{lem::bubbleprops_two} of Lemma \ref{lem::bubbleprops}, we have
  that $\boldsymbol{B}^V(\vecb{v}_h^V) \in
  \mathrm{BDM}_k(\mathcal{T}_V)$ and thus $\divergence(\boldsymbol{B}^V(\vecb{ v}_h^V)) \in
  \primalQh$, which gives 
  \begin{align*}
    \sum_{T\in\mathcal{T}_V}  (p - \primalph, \divergence (\boldsymbol{B}^V(\vecb{ v}_h^V)))_T &= \sum_{T\in\mathcal{T}_V}  (\pi^{\primalQh} p - \primalph - c_p, \divergence (\boldsymbol{B}^V(\vecb{ v}_h^V)))_T \\
    &\lesssim \| \pi^{\primalQh} p - \primalph - c_p\|_{\omega_V} \| (\vecb{ v}_h^V,\vecb{\hat v}_h^V)\|_{\vecb{X}_h^V},
  \end{align*}
  where we again used the continuity of $\boldsymbol{B}^V$. By the
  inf-sup condition of the primal Stokes dicretisation
  ($\pi^{\primalQh} p - \primalph - c_p$ has a zero mean value on
  $\omega_V$) on the local space \( \vecb{\bar V}_h \cap H^1_0(\omega_V)\) we have
  \begin{align*}
 \| \pi^{\primalQh} p- \primalph - c_p \|_{\omega_V} \lesssim \sup\limits_{\primalvh \in \vecb{\bar V}_h \cap H^1_0(\omega_V)} \frac{(\pi^{\primalQh} p - \primalph - c_p, \divergence(\primalvh))_{\omega_V} }{\| \nabla \primalvh \|_{\omega_V}}.
  \end{align*}
  Now, using that $\primalph$ is the discrete pressure solution we get
  \begin{align*}
    - (\primalph, \divergence(\primalvh))_{\omega_V} 
    &= (\vecb{f}, \mathcal{R}(\primalvh))_{\omega_V} - (\nu \nabla \primaluh, \nabla \primalvh)_{\omega_V} \\
    &= (-\nu \Delta \vecb{u} + \nabla p, \mathcal{R}(\primalvh))_{\omega_V} - (\nu \nabla \primaluh, \nabla \primalvh)_{\omega_V}.
  \end{align*}  
  Since $\divergence(\mathcal{R}(\primalvh)) \in \primalQh$, see \eqref{rec_two}, we get using integration by parts
  \begin{align*}
    (\nabla p, \mathcal{R}(\primalvh))_{\omega_V} 
    = - (p , \divergence(\mathcal{R}(\primalvh)))_{\omega_V}
    &=  - (\pi^{\primalQh} p, \divergence(\mathcal{R}(\primalvh)))_{\omega_V}\\
    &= - (\pi^{\primalQh} p, \divergence(\primalvh))_{\omega_V},
  \end{align*}
  and so in total (since $ (c_p, \divergence(\primalvh))_{\omega_V} =0$ by Gauss's theorem)
  \begin{align*} 
    (\pi^{\primalQh} p - \primalph - c_p, \divergence(\primalvh))_{\omega_V}
    &= -(\nu \Delta \vecb{u}, \mathcal{R}(\primalvh))_{\omega_V} - (\bar \sigma_h, \nabla \primalvh)_{\omega_V}\\
    &= -(\nu \Delta \vecb{u},\mathcal{R}(\primalvh) - \primalvh)_{\omega_V} + (\sigma - \bar \sigma_h, \nabla \primalvh)_{\omega_V},
  \end{align*}
  where we added and subtracted (including integration by parts)
  $(\sigma, \nabla \primalvh)_{\omega_V}$. By the properties of the
  reconstruction operator, the first integral can be bounded by 
  \begin{align} \label{eq::effrec}
    (-\nu \Delta \vecb{u}, \mathcal{R}(\primalvh) - \primalvh)_{\omega_V}  \lesssim \| (\id - \vecb{\pi}^{r-2}_{\omega_V})\nu  \Delta \vecb{u} \|_{\omega_V} h_V \| \nabla \primalvh \|_{\omega_V}, 
  \end{align}
  where $h_V$ denotes the diameter of the vertex patch $\omega_V$. Thus by the Cauchy Schwarz inequality we get the estimate
  \begin{align*}
    \| \pi^{\primalQh} p - \primalph - c_p \|_{\omega_V} \lesssim h_V\| (\id - \vecb{\pi}^{r-2}_{\omega_V})\nu  \Delta \vecb{u} \|_{\omega_V} + \del{ \sum_{T\in\mathcal{T}_V} \|(\sigma - \bar \sigma_h)\|^2_T}^{1/2},
  \end{align*}
  and so
  \begin{align*}
    \vecb{r}^V((\vecb{ v}^V_h,\vecb{\hat v}_h^V)) &\lesssim  \del{\sum_{T\in\mathcal{T}_V} \|(\sigma - \bar \sigma_h)\|^2_T + h_T^2 \| (\sigma - \bar \sigma_h)_{nt} \|^2_{\partial T}}^{1/2} \| (\vecb{v}_h^V,\vecb{\hat v}_h^V)\|_{\vecb{X}_h^V}  \\
    &+ h_V \| (\id - \vecb{\pi}^{r-2}_{\omega_V})\nu  \Delta \vecb{u} \|_{\omega_V}  \| (\vecb{ v}_h^V,\vecb{\hat v}_h^V)\|_{\vecb{X}_h^V}.
  \end{align*}
  This concludes the proof for the general case. Now assume that
  $\mathcal{R} = \id$, then we see that the additional term in
  \eqref{eq::effrec} vanishes which proves the stated result in the
  case where no reconstruction operator in the primal method
  \eqref{primalform} is included.
\end{proof}

\section{Numerical Examples}\label{sec:numerics} 
This section confirms
the theoretical results by some numerical examples. For the ease of
representation we introduce the following notation. 
The pressure-robust estimator of Theorem \ref{thm:EQestimator} is
denoted by $\eta$. Here, the flux $\eqflux$ either corresponds to the
solution $\eqflux^\text{GEQ}$ of the global problem \eqref{global:MCS}
or to the local equilibrated flux $\eqflux^\text{LEQ}$ given by
equation \eqref{localeqflux}. Further, we track the error estimator contributions
  \begin{align*}
    \eta_{\bold f}(\eqflux) &:= \nu^{-1}\| h_T^2(\id - \pi_{k-2})\operatorname{curl}(\bold f)\|,\\
    \eta_{\sigma}(\eqflux) &:= \nu^{-1}\|\operatorname{dev}(\eqflux - \bar \sigma_h)\|, \\
    \eta_{\divergence} &:= c_0^{-1}\|\divergence( \primaluh)\|. 
  \end{align*}
  Recall that Table \ref{infsuppairs} shows the different inf-sup
  stable velocity pressure pairs that we consider for the primal
  formulation \eqref{primalform}. 
  The order \(k = r\) corresponds to the order of the space
  \(\vecb{V}_h = \mathrm{RT}_{k}\), i.e. the order of the spaces used in
  the equilibration designs \eqref{global:MCS} and
  \eqref{localeqflux}. 
The adaptive mesh refinement loop is defined as usual by 
\begin{align*}
   \mathrm{SOLVE} \rightarrow \mathrm{ESTIMATE} \rightarrow \mathrm{MARK} \rightarrow \mathrm{REFINE} \rightarrow \mathrm{SOLVE} \rightarrow \ldots
\end{align*}
and employs the local contributions to the error estimator as
element-wise refinement indicators. In the marking step, an element \(T \in \mathcal{T}\) is marked for refinement if \(\eta(T) \geq \frac{1}{4} \max\limits_{K \in \mathcal{T}} \eta(K)\). The refinement step refines all marked elements plus further elements in a closure step to guarantee a regular triangulation.

In the case of the Scott--Vogelius (SV) finite element approximation, the adaptive algorithm includes two meshes: the macro element mesh $\mathcal{T}$ given by a standard triangulation, and the corresponding barycentrically refined triangulation (guaranteeing inf-sup stability of the SV element) denoted by $\mathcal{T}_{\textrm{bar}}(\mathcal{T})$. Again, an element \(T \in \mathcal{T}\) is marked if (mean value of the elements included in one macro element)
\begin{align*}
  \frac{1}{3} \sum\limits_{\substack{T' \in \mathcal{T}_{\textrm{bar}}\\ T' \cap T \neq \emptyset }} \mu(T') \ge \frac{1}{4} \max\limits_{K \in \mathcal{T}_{\textrm{bar}}} \eta(K).
\end{align*}
The refinement of $\mathcal{T}$ is done as described before. The final
mesh is then obtained by a global barycentric refinement step. Note
that although the macro element meshes are nested, their barycentric
refinements are in general not nested.    

 The implementation and numerical examples where performed with the finite element library NGSolve/Netgen \cite{ngsolve,netgen}, see also \url{www.ngsolve.org}.

  \begin{figure}
  \begin{center}
  \includegraphics[]{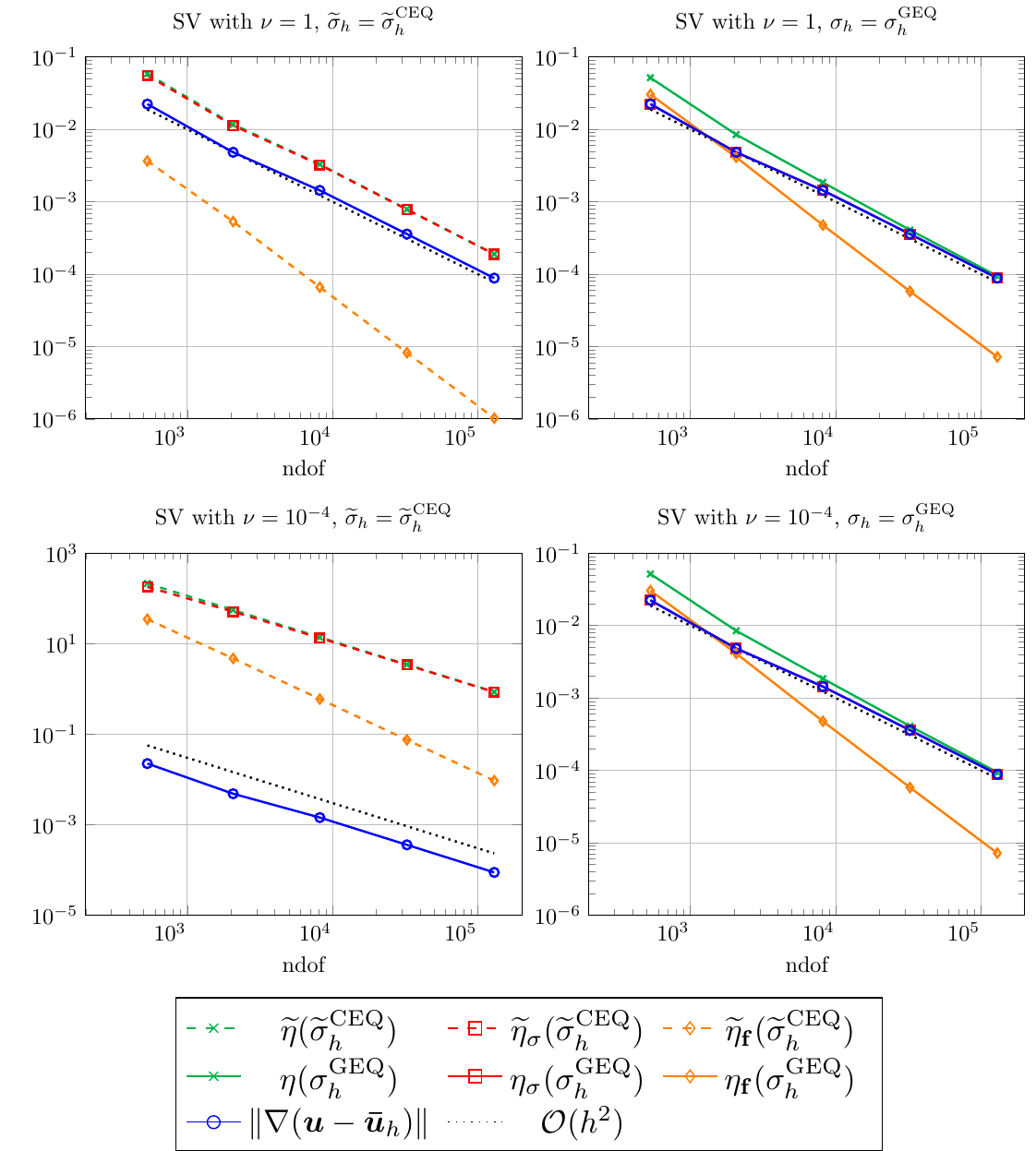}
  \vspace{5pt}
  \caption{Example from Section~\ref{curltwodimexample}: Convergence history of exact error and error
  estimator quantities on uniformly refined meshes for SV with
  \(\nu=1\) (top) and \(10^{-4}\) (bottom) and \(\sigma =
  \pseudostress_h^\text{CEQ}\) (left) and \(\sigma =
  \sigma_h^\text{GEQ}\) (right).} \label{fig:exone}
  \end{center}
\end{figure}

\begin{figure}
  \begin{center}
  \includegraphics[]{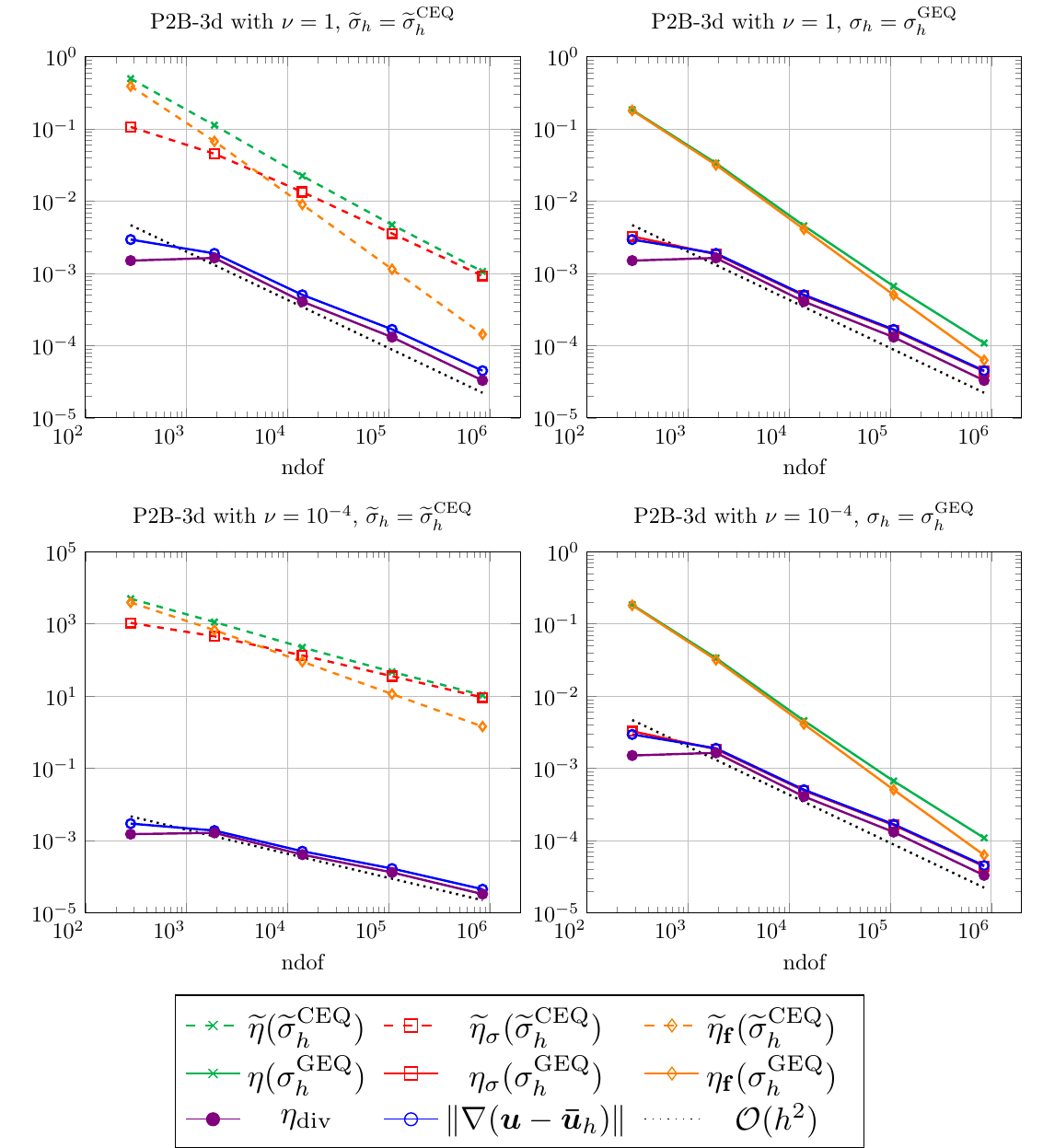}
  \vspace{5pt}
  \caption{Example from Section~\ref{curlthreedimexample}: Convergence history of exact error and error estimator quantities on uniformly refined meshes for P2B-3d with \(\nu=1\) (top) and \(10^{-4}\) (bottom) and \(\sigma = \pseudostress_h^\text{CEQ}\) (left) and \(\sigma = \sigma_h^\text{GEQ}\) (right).} \label{fig:exonethreed}
  \end{center}
\end{figure}

\begin{table}[h]
\begin{center}
\normalsize
\begin{tabular}{l@{~~}l@{~~}|@{~}c@{~}@{~}c@{~}@{~}c@{~}@{~}c@{~}@{~}c@{~}}
   & ref. level & 0 & 1 & 2 & 3 & 4\\
  \toprule 
  $\nu = 1$ & $\pseudostress_h = \pseudostress_h^\text{CEQ}$ & \num{2.615675242829881}& \num{2.4301510315073873}& \num{2.2889065097839194}& \num{2.198345006213172}& \num{2.147047729774112}\\
  $\nu = 1$ & $\sigma_h = \sigma_h^\text{GEQ}$ & \num{2.3043254083875664}& \num{1.7540568816450133}& \num{1.2856609599485287}& \num{1.1376152118271434}& \num{1.0683395985651158} \\
$\nu = 1$ & $\sigma_h = \sigma_h^\text{LEQ}$& \num{3.0835389922649927}& \num{2.4781802933350163}& \num{2.0063464495825576}& \num{1.8061490003959697}& \num{1.696949200054742}\\
  $\nu = 10^{-4}$ & $\pseudostress_h = \pseudostress_h^\text{CEQ}$ & \num{9526.38448060565}& \num{11515.790425222802}& \num{9660.55411067006}& \num{9633.060781469472}& \num{9750.079693089981}\\
  $\nu = 10^{-4}$ & $\sigma_h = \sigma_h^\text{GEQ}$ & \num{2.304325408365259}& \num{1.7540568815675341}& \num{1.2856609598912296}& \num{1.1376152117142277}& \num{1.0683395984819313}\\
  $\nu = 10^{-4}$ & $\sigma_h = \sigma_h^\text{LEQ}$ & \num{3.0835389920643097}& \num{2.4781802935858224}& \num{2.006346447176613}& \num{1.8061489934904533}& \num{1.6969491803509396}\\  
  \bottomrule
\end{tabular}
\caption{Efficiency indices in the Example from Section~\ref{curltwodimexample} on uniformly refined meshes and the SV element.}\label{exone::eff}
\end{center}
\end{table}
  
\subsection{Smooth example on unit square} \label{curltwodimexample}

First, we revisit the smooth example from
Section~\ref{sec::numericsclassical}. Figure \ref{fig:exone} presents
the convergence history of the error of the discrete Stokes solution
$\bar{\vecb{u}}_h$ measured in the $H^1$-semi norm using the SV
element with two different viscosities $\nu = 1$ (top) and $\nu =
10^{-4}$ (bottom) on uniformly refined meshes. The first important
observation is that the error plot for the pressure-robust error
estimator \(\eta(\sigma_h^\text{GEQ})\) looks exactly the same for \(\nu =
1\) and \(\nu = 10^{-4}\), while the classical estimator
\(\widetilde \eta(\pseudostress_h^\text{CEQ})\) is nowhere close to the exact error of
the pressure-robust Scott--Vogelius solution for \(\nu = 10^{-4}\)
as already observed in Section~\ref{sec::numericsclassical}. As
expected, the error estimator scales with \(\nu^{-1}\) and so does its
efficiency index.

Table~\ref{exone::eff} lists the efficiency indices on the different
refinement levels also for the pressure-robust local variant of our
error estimator. One can see that the error estimator for
\(\eta(\sigma_h^\text{GEQ})\) even is asymptotically exact, while the local
variant is not, but still attains very good efficiency indices around
\(2\). We want to mention again that the novel error bounds, unfortunately,
contain unknown constants \(c_1\) and \(c_2\) which were evaluated by
\(c_1 = c_2 = 1\). However, they only appear in front of
\(\eta_{\vecb{f}}\) which is a higher order term (see Figure~\ref{fig:exone} again).


\subsection{Smooth example on unit cube} \label{curlthreedimexample}
The second example extends the previous example onto the
unit cube \(\Omega = (0,1)^3\) by prescribing the solution
\begin{align*}
  \vecb{u}(x,y) := \curl \left(\xi,\xi,\xi\right) \quad \text{and} \quad  p(x,y) := x^5 + y^5 + z^5 - 1/2
\end{align*}
with the potential $\xi = x^2(1-x)^2y^2(1-y)^2z^2(1-z)^2$ and with
matching right-hand side \(\vecb{f} := - \nu \Delta \vecb{u} + \nabla
p\) for variable viscosity \(\nu\).

Figure \ref{fig:exonethreed} presents the convergence history of the
error of the discrete Stokes solution $\bar{\vecb{u}}_h$ measured in
the $H^1$-semi norm using the P2B-3d element with two different
viscosities $\nu = 1$ (top) and $\nu = 10^{-4}$ (bottom) on uniformly
refined meshes. The observations are similar to the ones in the two
dimensional case which validates our results also for the case $d =
3$. Since the right-hand side $\vecb{f}$ is a
polynomial of higher order compared to the two dimensional example,
the oscillation terms $\eta_{\vecb{f}}(\eqflux^\text{GEQ}), \widetilde \eta_{\vecb{f}}(\pseudostress_h^\text{CEQ})$ are much
larger and hence more dominating on coarser levels. 

\subsection{L-shaped domain example} \label{lshapeexample} The final
example from \cite{MR993474} is defined on the L-shaped domain
\(\Omega := (-1,1)^2 \setminus \left((0,1) \times (-1,0)\right)\). The
velocity $\vecb{u}$ and pressure $p_0$ now satisfy \(-\nu \Delta
\vecb{u} + \nabla p_0 = 0\), and read as (given in polar coordinates
with radius $R$ and angle $\varphi$)
\begin{align*}
 \vecb{u}(R,\varphi)
& :=R^\alpha
\begin{pmatrix}
(\alpha+1)\sin(\varphi)\psi(\varphi) + \cos(\varphi)\psi^\prime(\varphi)
\\
-(\alpha+1)\cos(\varphi)\psi(\varphi) + \sin(\varphi)\psi^\prime(\varphi)
\end{pmatrix}^T,\\
 p_0 &:= \nu R^{(\alpha-1)}((1+\alpha)^2 \psi^\prime(\varphi)+\psi^{\prime\prime\prime}(\varphi))/(1-\alpha)
\end{align*}
with
\begin{multline*}
\psi(\varphi) :=
1/(\alpha+1) \, \sin((\alpha+1)\varphi)\cos(\alpha\omega) - \cos((\alpha+1)\varphi)\\
- 1/(\alpha-1) \, \sin((\alpha-1)\varphi)\cos(\alpha\omega) + \cos((\alpha-1)\varphi),
\end{multline*}
and \(\alpha = 856399/1572864 \approx 0.54\), \(\omega = 3\pi/2\). To
have a nonzero right-hand side we add the pressure \(p_+ :=
\sin(xy\pi)\), i.e. \(p := p_0 + p_+\) and \(\vecb{f} :=
\nabla(p_+)\). Note that since \(\vecb{f}\) is a gradient we have
\(\eta_{\vecb{f}} = 0\) in this example.

  \begin{figure}
  \begin{center}
  \includegraphics[]{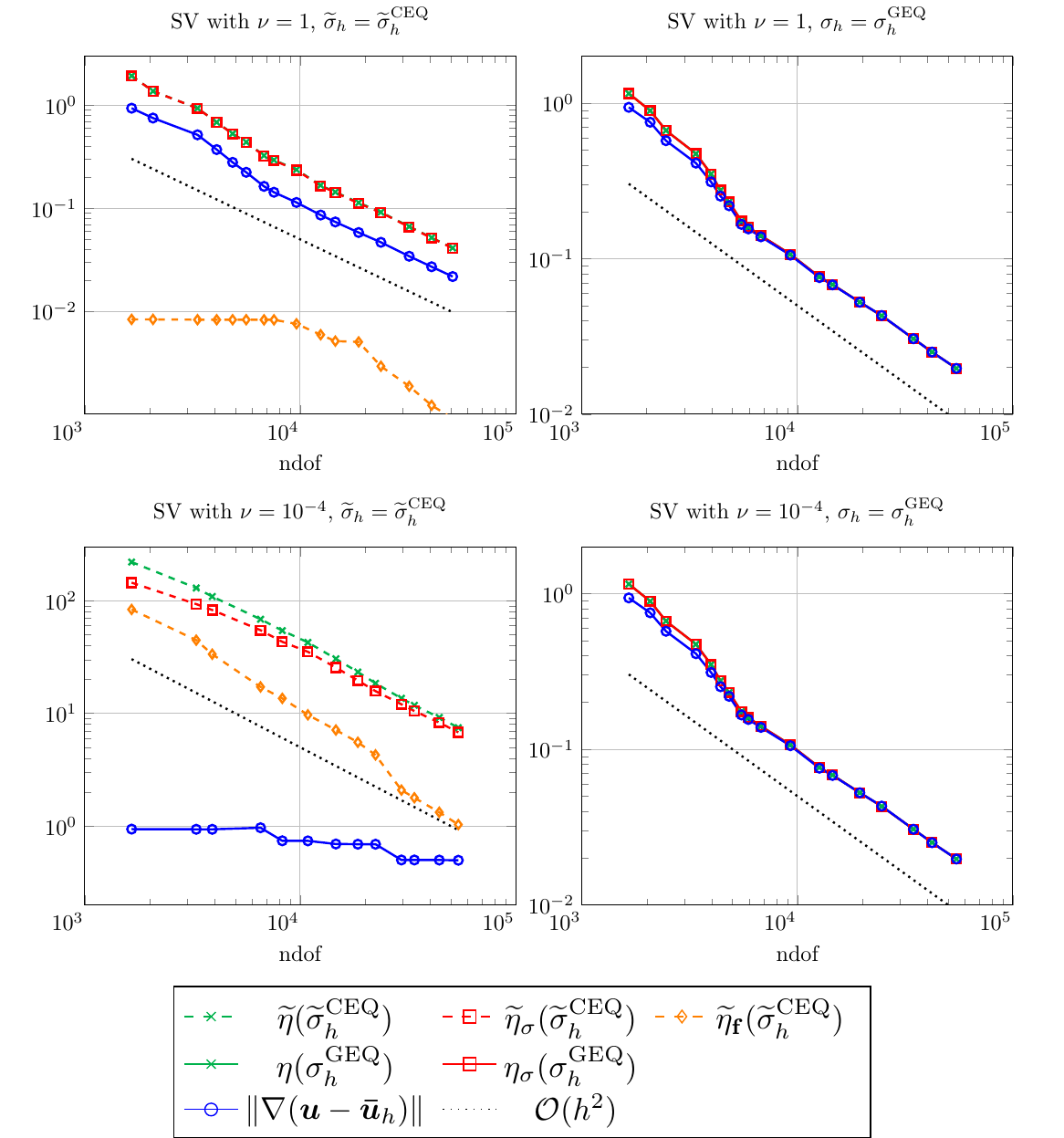}
  \vspace{5pt}
  \caption{Example from Section~\ref{lshapeexample}: Convergence history of exact error and error estimator quantities on adaptively refined meshes for SV with \(\nu=1\) (top) and \(10^{-4}\) (bottom) and \(\sigma_h = \pseudostress_h^\text{CEQ}\) (left) and \(\sigma_h = \sigma_h^\text{GEQ}\) (right).} \label{fig:extwo}
  \end{center}
\end{figure}

\begin{table}[h]
\begin{center}
\normalsize
\begin{tabular}{l@{~~}l@{~~}|@{~~}c@{~~}@{~~}c@{~~}@{~~}c@{~~}@{~~}c@{~~}@{~~}c@{~~}}
   & ref. level & $\textrm{ref}_{\textrm{tot}}-4$ & $\textrm{ref}_{\textrm{tot}}- 3$ &$\textrm{ref}_{\textrm{tot}} - 2$ & $\textrm{ref}_{\textrm{tot}}- 1$ & $\textrm{ref}_{\textrm{tot}}$\\
  \toprule 
$\nu = 1$ & $\pseudostress_h = \pseudostress_h^\text{CEQ}$ & \num{1.950869405027048}& \num{1.9619715446525452}& \num{1.9442290953329584}& \num{1.9184250005856938}& \num{1.8963380348103933}\\
  $\nu = 1$ & $\sigma_h = \sigma_h^\text{GEQ}$ & \num{1.0015155446742894}& \num{0.9964303972537422}& \num{1.0010267585249284}& \num{1.0017410760482752}& \num{0.9998534402013637}\\
$\nu = 1$ & $\sigma_h = \sigma_h^\text{LEQ}$& \num{2.3929776604352044}& \num{2.3857193360663356}& \num{2.3586137791846435}& \num{2.364500282721352}& \num{2.3708925891930854} \\
  $\nu = 10^{-4}$ & $\pseudostress_h = \pseudostress_h^\text{CEQ}$ & \num{26.788404636152507}& \num{27.171719978255144}& \num{23.567756345368593}& \num{18.402431204960283}& \num{15.004169081137151}\\
  $\nu = 10^{-4}$ & $\sigma_h = \sigma_h^\text{GEQ}$ & \num{1.0015155367164859}& \num{0.9964303855264838}& \num{1.0010267604443712}& \num{1.0017410682414458}& \num{0.9998534350439882}\\
  $\nu = 10^{-4}$ & $\sigma_h = \sigma_h^\text{LEQ}$ & \num{2.3929776603221646}& \num{2.3857193359504656}& \num{2.35861377913446}& \num{2.3645002827031982}& \num{2.370892589029835}\\  
  \bottomrule
\end{tabular}
\caption{Efficiency indices in the Example from Section~\ref{lshapeexample} on the last five adaptively refined meshes using the SV element. Here $\textrm{ref}_{\textrm{tot}}$ denotes the total number of refinement steps.}\label{extwo::eff}
\end{center}
\end{table}

Figure~\ref{fig:extwo} shows the convergence history of the exact
error and the error estimators based on the classical equilibrated
fluxes \(\pseudostress_h^\text{CEQ}\) and the pressure-robust fluxes
\(\sigma_h^\text{GEQ}\) on adaptively refined meshes where the
refinement indicators are steered by the local contributions of the
estimators. For \(\nu = 1\) both estimators are efficient, the
pressure-robust one is even asymptotically exact, and all convergence
rates are optimal. For \(\nu = 10^{-4}\) the numbers and meshes for
the pressure-robust estimator are exactly the same (which is expected,
since the discrete velocity did not change), but the adaptive meshes
for the classical estimator do not refine the corner singularity and
therefore fail to reduce the velocity error optimally. Here, the refinement
indicators only see the dominating pressure error and mark accordingly
to reduce the pressure error. Adaptation to the corner singularity
only starts when the velocity
error and the pressure error scaled with \(\nu^{-1}\) are on par.
The slow decrease of the efficiency indices in Table~\ref{extwo::eff} can be explained by the bestapproximation error reduction of the smooth pressure.
Consequently, the exact velocity error on the final mesh obtained with
refinement indicators based on \(\pseudostress_h^\text{CEQ}\) is still
larger by more than one order of magnitude compared to the error on the final mesh obtained with refinement indicators based on \(\sigma_h^\text{GEQ}\).
These observations also support the discussion in Remark~\ref{rem:classical_efficiency}.

  \begin{figure}
  \begin{center}
  \includegraphics[]{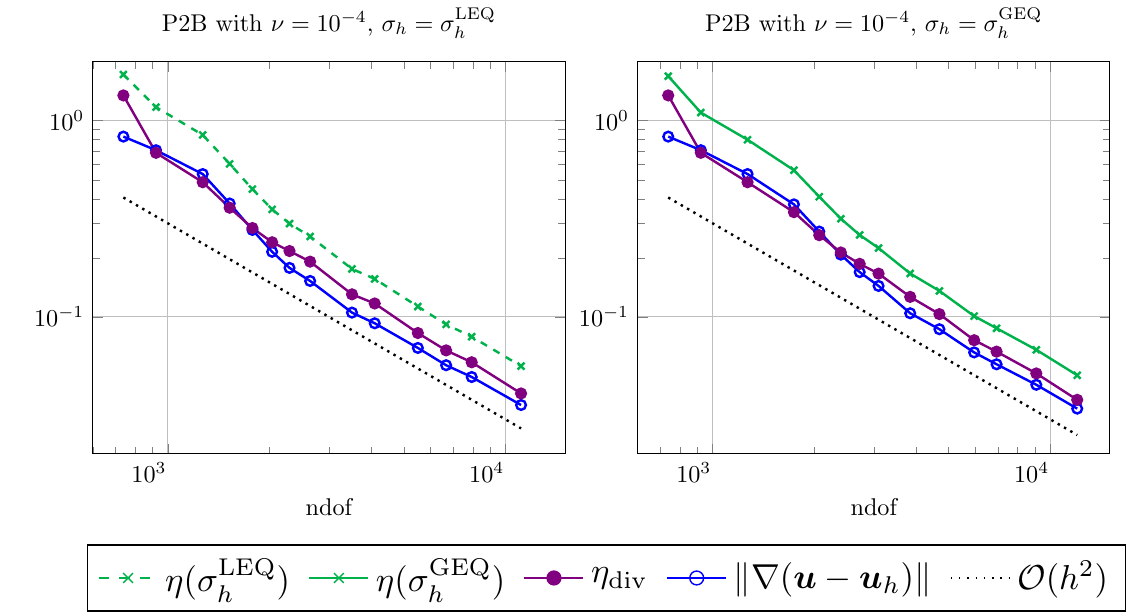}
  \vspace{5pt}
  \caption{Example from Section~\ref{lshapeexample}: Convergence history of exact error and error estimators on adaptively refined meshes for P2B with \(\nu = 10^{-4}\) and \(\sigma_h = \sigma_h^\text{LEQ}\) (left) and \(\sigma_h = \sigma_h^\text{GEQ}\) (right).} \label{fig:extwo_P2B}
  \end{center}
\end{figure}

  \begin{figure}
  \begin{center}
  \includegraphics[]{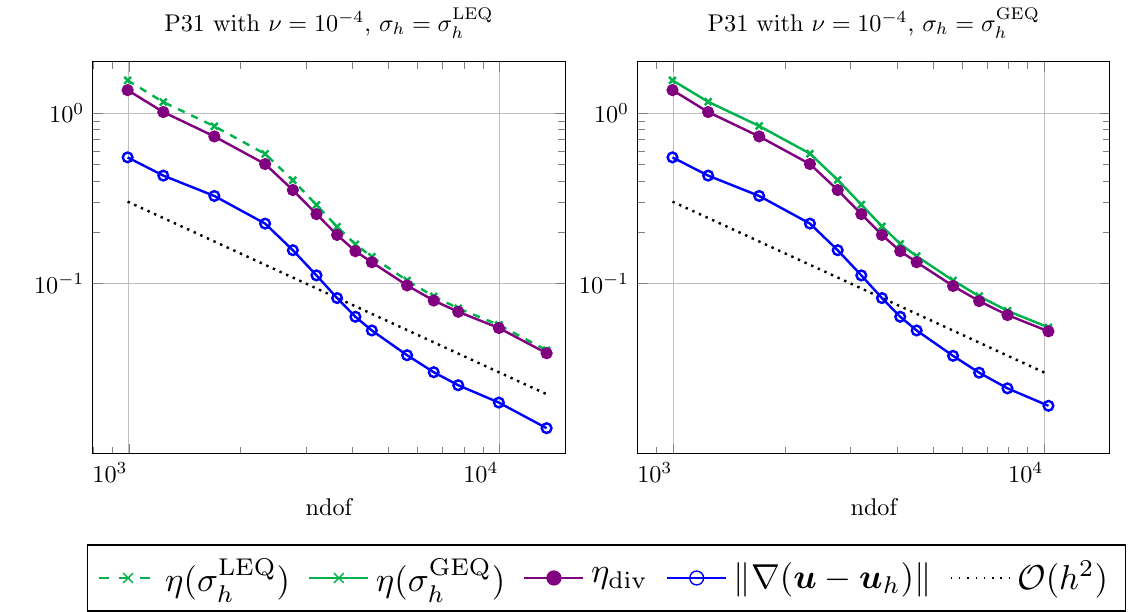}
  \vspace{5pt}
  \caption{Example from Section~\ref{lshapeexample}: Convergence history of exact error and error estimators on adaptively refined meshes for P31 with \(\nu = 10^{-4}\) and \(\sigma_h = \sigma_h^\text{LEQ}\) (left) and \(\sigma_h = \sigma_h^\text{GEQ}\) (right).} \label{fig:extwo_P31}
  \end{center}
\end{figure}

  \begin{figure}
  \begin{center}
  \includegraphics[]{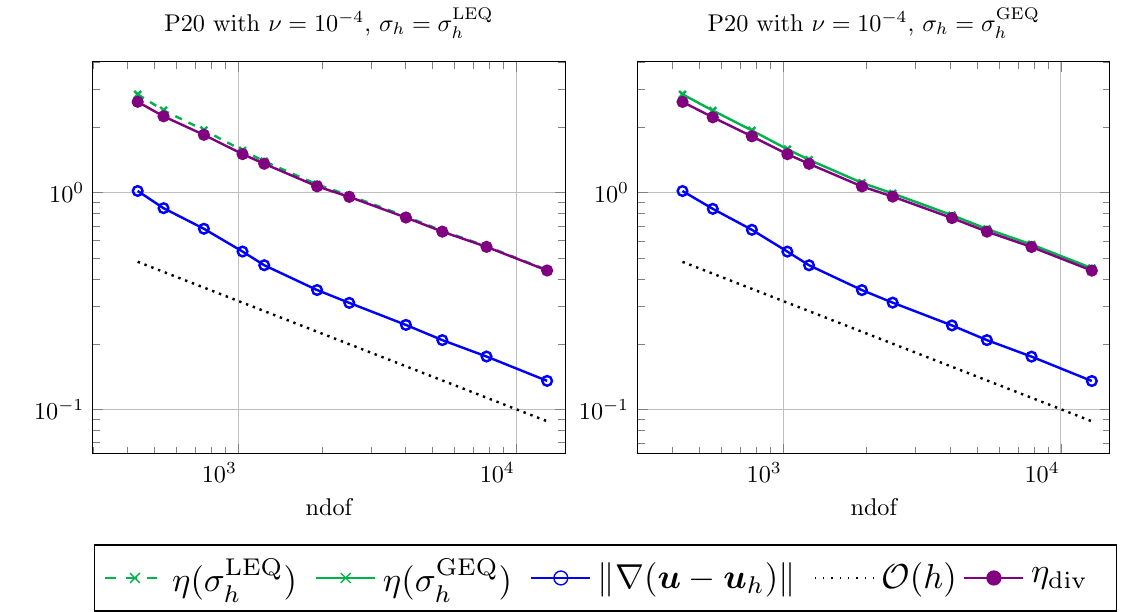}
  \vspace{5pt}
  \caption{Example from Section~\ref{lshapeexample}: Convergence history of exact error and error estimators on adaptively refined meshes for P20 with \(\nu = 10^{-4}\) and \(\sigma_h = \sigma_h^\text{LEQ}\) (left) and \(\sigma_h = \sigma_h^\text{GEQ}\) (right).} \label{fig:extwo_P20}
  \end{center}
\end{figure}

Figures~\ref{fig:extwo_P2B}-\ref{fig:extwo_P20} display results for
the three other methods P2B, P31 and P20 for the local and global
variant of our pressure-robust error estimator. Since, the discrete
velocity and the error estimator is independent of \(\nu\), we only
show the results for \(\nu = 10^{-4}\). Note that these methods are
not divergence-free but pressure-robust due to their reconstruction
operator in the right-hand side. However, this causes
\(\divergence(\vecb{u}_h) \neq 0\) and hence the contribution
\(\eta_\text{div}\) appears here which also requires a lower bound
for the inf-sup constant $c_0$. Here, we take the value \(c_0 = 0.3\)
from \cite{STOYAN1999243}. Unfortunately, this has a significant impact on the
efficiency of the error estimator that is largest for P20 and smallest
for P2B leading to still very small efficiency indices between 1.5 and
3 for both the local and the global equilibration error estimators.

\section*{Acknowledgements}
Philip L. Lederer has been funded by the Austrian Science Fund (FWF)
through the research program ``Taming complexity in partial
differential systems'' (F65) - project ``Automated discretization in
multiphysics'' (P10).

\bibliographystyle{plain}
\bibliography{lit}

\end{document}